%% file: arxiv.tex
\documentclass[a4paper,11pt,colorlinks]{article}
\usepackage[margin=3cm]{geometry}
\usepackage{authblk}
\usepackage{amsmath}
\usepackage{amsthm}
\usepackage{amsfonts}
\usepackage{amssymb}
\usepackage{enumitem}
\usepackage{doi}
\usepackage{authblk}
\usepackage{xspace}
\usepackage{xparse}

\usepackage[numbers,sort]{natbib}

\usepackage{mathtools}
\usepackage{todonotes}
\usepackage{thm-restate}
\usepackage{hyperref}

\graphicspath{{./graphics/}}

\DeclareMathOperator{\tw}{tw} 
\DeclareMathOperator{\Des}{Des} 
\DeclareMathOperator{\maj}{maj} 
\DeclareMathOperator{\occ}{occ} 

\DeclareMathOperator{\Av}{Av}
\DeclareMathOperator{\Grid}{Grid}

\newcommand{\TOTO}{\textsf{TOTO}\xspace}
\newcommand{\TOOB}{\textsf{TOOB}\xspace}
\newcommand{\TOLO}{\textsf{TOLO}}
\newcommand{\TO}{\textsf{TO}}
\newcommand{\TOW}{\textsf{TOW}}
\newcommand{\TOG}{\textsf{TOG}}
\newcommand{\TOLG}{\textsf{TOLG}}

\newcommand{\bfx}{\mathbf{x}}
\newcommand{\bfa}{\mathbf{a}}
\newcommand{\bfb}{\mathbf{b}}

\ExplSyntaxOn

\NewDocumentCommand{\definealphabet}{mmmm}
{
  \int_step_inline:nnn { `#3 } { `#4 }
  {
    \cs_new_protected:cpx { #1 \char_generate:nn { ##1 }{ 11 } }
    {
      \exp_not:N #2 { \char_generate:nn { ##1 } { 11 } }
    }
  }
}

\ExplSyntaxOff

\makeatletter
\NewDocumentCommand{\tagx}{om}{%
  \IfNoValueTF{#1}
   {
    \refstepcounter{equation}(\theequation)\label{#2}%
   }
   {
    (#1)\def\@currentlabel{#1}\label{#2}%
   }%
}
\makeatother

\newenvironment{proofclaim}{\noindent{\em Proof of the claim.}}{\qedclaim}
\newcommand{\qedclaim}{\hfill $\diamond$ \medskip}

\definealphabet{c}{\mathcal}{A}{Z}

\DeclareMathOperator{\qd}{qd} 

\mathchardef\mhyphen="2D

\def\dd{\kern.4ex\mbox{\raise.4ex\hbox{{\rule{.35em}{.12ex}}}}\kern.4ex}

\newcommand{\ra}{\triangleright}
\newcommand{\ta}{\vartriangle}
\newcommand{\la}{\triangleleft}
\newcommand{\da}{\triangledown}

\theoremstyle{plain}
\newtheorem{theorem}{Theorem}[section]
\newtheorem{observation}[theorem]{Observation}

\newtheorem{claim}[theorem]{Claim}
\newtheorem{proposition}[theorem]{Proposition}

\theoremstyle{definition}
\newtheorem{example}[theorem]{Example}

\begin{document}

\title{Monadic Second-Order Logic of Permutations}

\setcounter{Maxaffil}{0}
\renewcommand\Affilfont{\itshape\small}
\author[1]{Vít Jelínek\thanks{Supported by project 23-04949X of the Czech Science Foundation.}}
\author[2]{Michal Opler}
\affil[1]
{Computer Science Institute, Charles University, Czechia}
\affil[2]
{Czech Technical University in Prague, Czechia}
%
\date{}
\maketitle              

\begin{abstract}
Permutations can be viewed as pairs of linear orders, or more formally as models over a signature consisting of two binary relation symbols.
This approach was adopted by Albert, Bouvel and Féray, who studied the expressibility of first-order logic in this setting.
We focus our attention on monadic second-order logic.

Our results go in two directions.
First, we investigate the expressive power of monadic second-order logic.
We exhibit natural properties of permutations that can be expressed in monadic second-order logic but not in first-order logic.
Additionally, we show that the property of having a fixed point is inexpressible even in monadic second-order logic.

Secondly, we focus on the complexity of monadic second-order model checking.
We show that there is an algorithm  deciding if a permutation $\pi$ satisfies a given monadic second-order sentence $\varphi$ in time $f(|\varphi|, \tw(\pi)) \cdot n$ for some computable function $f$ where $n = |\pi|$ and $\tw(\pi)$ is the tree-width of $\pi$. 
On the other hand, we prove that the problem remains hard even when we restrict the 
permutation $\pi$ to a fixed hereditary class $\mathcal{C}$ with mild assumptions on $\mathcal{C}$.
\end{abstract}

\section{Introduction}
\label{sec:introduction}

\input{intro.tex}

\section{Preliminaries}
\label{sec:prelim}

A \emph{permutation $\pi$ of length $n$} is a sequence $\pi_1, \dots, \pi_n$ that contains each element of the set $[n] = \{1,\dots,n\}$ exactly once.
Note that we omit the punctuation when writing out short permutations explicitly, e.g., we write $3142$ instead of $3,1,4,2$.
It is often beneficial to view permutations as geometric objects -- we associate to each permutation $\pi$ the \emph{permutation diagram} $S_\pi = \{ (i, \pi_i) \mid i \in [n]\}$.
Observe that no two points from $S_\pi$ share the same $x$- or $y$-coordinate.
We say that such a set is in \emph{general position}.

For a point $p \in \mathbb{R}^2$ in the plane, we let $p.x$ denote its (first) $x$-coordinate and, $p.y$ its (second) $y$-coordinate.
We say that two finite sets $S,T \subset \mathbb{R}^2$ in general position are \emph{isomorphic} if there exists a bijection $f\colon S \to T$ such that for any pair of points $p,q \in S$ we have $p.x < q.x$ if and only if $f(p).x < f(q).x$, and $p.y < q.y$ if and only if $f(p).y < f(q).y$.

A permutation $\tau$ \emph{contains} a permutation $\pi$ if the diagram $S_\tau$ contains a 
subset isomorphic to~$S_\pi$.
Otherwise, we say that $\tau$ \emph{avoids} $\pi$.
A \emph{permutation class} is a hereditary set $\cC$ of permutations, i.e., whenever $\pi \in \cC$ and $\pi$ contains $\sigma$ then also $\sigma \in \cC$.
The easiest way of obtaining a permutation class is to take all the permutations avoiding a fixed permutation~$\sigma$.
We let $\Av(\sigma)$ denote the class of all $\sigma$-avoiding permutations.
In particular, $\Av(21)$ is the class of increasing permutations and $\Av(12)$ is the class of decreasing permutations.

\paragraph{Operations acting on permutations.}
Let $\pi$ be a permutation of length~$n$.
The \emph{reverse of $\pi$} is the permutation $\pi^r = \pi_n, \pi_{n-1}, \dots, \pi_1$ and the \emph{complement of $\pi$} is the permutation $\pi^c = n+1 - \pi_1, n+1-\pi_2, \dots, n+1-\pi_n$.
In other words, the reverse is obtained by mirroring $\pi$ horizontally while complement corresponds to mirroring $\pi$ vertically. 

Let $\sigma$ be a permutation of length $n$ and let $\tau_1, \dots, \tau_n$ be a sequence of permutations.
The \emph{inflation of $\sigma$ by $\tau_1, \dots,\tau_n$}, denoted $\sigma[\tau_1, \dots, \tau_n]$, is the permutation isomorphic to the point set obtained by replacing each point $(i, \sigma_i) \in S_\sigma$ with a tiny copy of $S_{\tau_i}$.
A permutation is \emph{simple}, if it cannot be obtained from strictly smaller permutations by an inflation.
For instance, the permutation 25314 is simple, while 25341 is not, since it can be obtained, e.g., as $231[1,312,1]$.

As a specific case of inflation, the \emph{(direct) sum} of two permutations $\sigma$ and~$\tau$, denoted
$\sigma\oplus\tau$, is the permutation $12[\sigma, \tau]$ while the \emph{skew sum} of $\sigma$ and $\tau$,
denoted $\sigma\ominus\tau$,
is the permutation $21[\sigma, \tau]$.
Finally, the \emph{separable permutations} are the permutations that can be created from the singleton permutation of size~1 by direct sums and skew sums; it is known~\cite{Bose1998} that these are precisely the permutations avoiding the patterns $2413, 3142$.

\paragraph*{Tree-width.}
Let us first introduce the standard definition of tree-width as a graph parameter.
A \emph{tree decomposition} of a graph $G$ is a pair $(T,\beta)$, where $\beta\colon  V(T) \to 2^{V(G)}$ assigns a \emph{bag} $\beta(p)$ to each vertex of T, such that
\begin{itemize}
  \item for every vertex $v$ of $G$, there exists $p \in V(T)$ such that $v \in
  \beta(p)$,
  \item for every edge $uv$ in $G$, there exists $p \in V(T)$ such that $u, v \in
  \beta(p)$, and
  \item for every vertex $v$ of $G$, the set $\lbrace p \in V(T): v \in \beta(p)
  \rbrace$ induces a connected subtree of $T$.
\end{itemize}
The \emph{width} of the tree decomposition is the maximum of $\beta(p) - 1$ over all $p \in V(T)$.
The \emph{tree-width} $\tw(G)$ of a graph $G$ is the minimum width of a tree decomposition of $G$.

The tree-width of a permutation $\pi$ is then defined by taking a tree-width of a certain graph encoding the structure of $\pi$.
The \emph{incidence-graph} $G_\pi$ of a permutation $\pi$ is the graph whose vertices 
are the $n$ points of $S_\pi$, and a point $p$ is connected to every point $q$ such that $|p.x - 
q.x| = 1$ or $|p.y - q.y| = 1$.
Less formally, the graph $G_\pi$ is a union of two paths, one of them visiting the points of $\pi$ in left-to-right order, and the other in top-to-bottom order.
The \emph{tree-width} of $\pi$, denoted by $\tw(\pi)$, is simply the tree-width of~$G_\pi$.

\paragraph*{FO and MSO logic.}
A \emph{signature} is a set of relation and function symbols, each associated with a non-negative integer, called \emph{arity}.
We restrict our attention to signatures consisting purely of relation symbols.
We are particularly interested in the signature $\cS_\TO$ that consists of two binary relation symbols $<_1$ and $<_2$.
These symbols intend to convey the ordering of points along the $x$- and $y$-axes.
A \emph{structure} of a signature $\cS$ is a pair $\cM = (A, I)$ where $A$ is an arbitrary set, 
called \emph{domain}, and $I$ describes an interpretation of the symbols in $\cS$ on $A$.
To be more precise, the interpretation $I(R)$ of a relation symbol $R$ with arity $k$ is a subset 
of~$A^k$.
For succinctness, we shall also denote the structures of $\cS_\TO$ simply as triples $\cM = (A, \prec^A_1, \prec^A_2)$ where $\prec^A_1$ and $\prec^A_2$ are interpretations of the symbols $<_1$ and $<_2$.
Observe that any permutation $\sigma$ can be naturally seen as a structure $(S_\sigma, \prec_1, \prec_2)$ of $\cS_\TO$ where $\prec_1$ and $\prec_2$ are the natural orders given by the $x$- and $y$-coordinates of points.

An \emph{atomic formula} over a signature $\cS$ is either an equality predicate ($x = y$), or a predicate $R(a_1, \ldots, a_k)$ for an arbitrary relation symbol $R$ of $\cS$ with arity $k$.
\emph{First-order (FO) formulas}, usually denoted by Greek letters, are formed inductively from atomic formulas and logical symbols.
In particular, a first-order  formula is either (i) an atomic formula, (ii) a negation of an FO formula ($\neg \varphi$), (iii) a conjuction ($\varphi_1 \land \varphi_2$), disjunction ($\varphi_1 \lor \varphi_2$), implication ($\varphi_1 \rightarrow \varphi_2$) or equivalence ($\varphi_1 \leftrightarrow \varphi_2$) of two FO formulas, or (iv) an existential ($\exists x\, \varphi$) or universal ($\forall x\, \varphi$) quantification of a FO formula.
A first-order (FO) \emph{sentence} is then any FO formula that has no free variables, or put differently, a formula whose variables are all quantified.

A structure $\cM = (A, I)$ \emph{satisfies} the sentence $\varphi$ if $\varphi$ evaluates to ``True'' when we interpret the variables as elements of the domain $A$ and the symbols of $\cS$ according to $I$.
We denote this by $\cM \models \varphi$.
In the case of $\cS_\TO$, we additionally allow ourselves to write $\sigma \models \varphi$ where $\sigma$ is the permutation associated to $\cM$.
Furthermore, the notion of satisfiability can be extended to formulas with free variables, given an assignment of its free variables to values from the domain.
We denote by $\varphi(\mathbf{x})$ an FO formula with free variables $\mathbf{x} = (x_1, \dots, x_k)$.
For a structure $\cM = (A,I)$ and a $k$-tuple $\mathbf{a} = (a_1, \dots, a_k) \in A^k$, we say that $(\cM, \mathbf{a})$ \emph{satisfies} $\varphi(\mathbf{x})$ if $\varphi(\mathbf{x})$ evaluates to ``True'' as before when we additionally interpret each free variable $x_i$ as $a_i$.
We denote this by $(\cM, \mathbf{a}) \models \varphi$.

Finally, we can formally define the notion of a theory.
A \emph{theory} is a set of (FO) sentences, which are called the \emph{axioms} of the theory.
A \emph{model} of a theory is any structure that satisfies all the axioms of the theory.
The \emph{Theory of Two Orders (\TOTO)} requires that both $<_1$ and $<_2$ are linear orders.
It is obvious that this requirement can be straightforwardly stated as FO sentences.
For more details on \TOTO{}, we refer to Albert et al.~\cite{Albert2020}.

\emph{Monadic second-order (MSO) formulas} extend FO formulas by allowing set variables (denoted by capital letters) and quantification over them.
Formally, MSO formula over a signature $\cS$ is formed inductively using the same operations as FO formulas with the addition of existential ($\exists X \, \varphi$) and universal ($\forall X \,\varphi$) set quantifications, and set membership predicates ($x \in X$).
As before, a \emph{monadic second-order sentence} is any MSO formula that has all its variables (representing both elements and sets) quantified.
The satisfiability of an MSO formula $\varphi$ by a permutation $\sigma$ is defined accordingly.

\section{Expressive Power of MSO}\label{sec:expr}

Our first goal is to explore the power of monadic second-order logic to express various properties of permutations.
First, we focus on modular counting.
Secondly, we present natural properties of permutations that are inexpressible in first-order logic but expressible in monadic second-order logic.
And finally, we show that the property of having a fixed point is inexpressible even in monadic second-order logic.


\subsection{Modular Counting}

A commonly studied extension of MSO logic is to allow predicates expressing the cardinality of sets modulo any (fixed) integer.
Formally, \emph{Counting monadic-second order (CMSO) formulas} are obtained by extending the definition of MSO formulas with an additional atomic formula $card_{q,r}(X)$ for any integer $r \ge 2$ and $q \in \{0, \dots, r-1\}$ that is satisfied if and only if $|X| \equiv q \pmod{r}$.
The expressiveness of CMSO logic has been studied, e.g., in~\cite{Courcelle, Courcelle96, Elberfeld2016, Ganzow2008}.

It is well known that CMSO formulas can be expressed in MSO in any structure equipped with a linear order.
Note that this is not true in general as for example, the simple property of having an even number of vertices is not MSO-definable on graphs.
It follows that any predicate $card_{q,r}(X)$ can be expressed by an MSO formula in \TOTO and thus, for every CMSO sentence in \TOTO there exists an equivalent MSO sentence in \TOTO. 

In the presence of a linear order, the same technique in fact yields considerably more than the predicates $card_{q,r}(X)$ alone.
In particular, we proceed to show that for any MSO formula $\varphi$ in $\TOTO$ with $k$ free (element) variables, we can express the number (modulo fixed integer) of all $k$-tuples from the domain satisfying $\varphi$.
This will allow us to show that MSO sentences in \TOTO are capable of defining modular constraints on various permutation statistics.
Just like the observation of Courcelle above, none of this is specific to \TOTO: Proposition~\ref{pro:tuple-cmso} holds verbatim in any theory of ordered structures, since it suffices to replace every occurrence of $<_1$ with the linear order $<$ in the proof below.
We state it for \TOTO only because that is the setting we need.

\begin{proposition}\label{pro:tuple-cmso}
Let $\varphi(\mathbf{x})$ be an MSO formula in \TOTO with $k$ free element variables.
For every $r \in \mathbb{N}$ and $q \in \{0,\dots, r-1\}$, there exists an MSO sentence $card_{q,r}^\varphi$ in \TOTO of length $r^{O(k + 1)} \cdot |\varphi|$ such that $\pi \models card_{q,r}^\varphi$ if and only if
\[\left|\{ \bfa \mid \bfa \in \pi^k \land (\pi, \bfa) \models \varphi(\bfx)\}\right| \equiv q \pmod{r}.\]
\end{proposition}
\begin{proof}
	We shall inductively define for each $\ell \in \{0, \dots, k\}$ and $s \in \{0, \dots, r-1\}$ a formula $\psi^\ell_s(\mathbf{x})$ with $\ell$ free variables $ \mathbf{x} = (x_1, \dots, x_\ell)$ such that $(\pi, \mathbf{a}) \models \psi^\ell_s(\bfx)$ if and only if
	\[\left|\{ \bfb \mid \bfb \in \pi^{k-\ell} \land (\pi, \bfa \circ \bfb) \models \varphi(\bfx)\}\right| \equiv s \pmod{r}\]
	where $\mathbf{a} \circ \mathbf{b}$ denotes conatenation of the tuples $\mathbf{a}$ and $\mathbf{b}$.
	In other words, $\psi^\ell_s$ verifies that the number of ways to complete its $\ell$ free variables to a $k$-tuple that models~$\varphi$ is congruent to $s$ modulo $r$.
	It then suffices to set $card_{q,r}^\varphi = \psi^0_q$.
		
	The definition of $\psi^k_s$ is immediate.
	Observe that the set $\{ \mathbf{b} \mid \mathbf{b} \in \pi^{0} \land (\pi, \mathbf{a} \circ \mathbf{b}) \models \varphi(\bfx)\}$ contains at most one element, the empty tuple. 
	Thus, we set
	\[ \psi^k_s(\mathbf{x}) = \begin{cases}
		\neg \varphi(\mathbf{x}) &\text{if $s = 0$,}\\
		\varphi(\mathbf{x}) &\text{if $s = 1$, and}\\
		\bot &\text{otherwise.}

	\end{cases}\]

	Now let $\ell < k$ and assume that we have constructed $\psi^{\ell'}_s$ for every $\ell' > \ell$ and $s \in \{0, \dots, r-1\}$.
	We set
	\begin{multline*}
		\psi^\ell_s (\mathbf{x}) = \exists X_0 \exists X_1 \dots \exists X_{r-1} \exists x_{\mathit{first}} \exists x_{\mathit{last}} \\
		\left(
			\begin{aligned}
				 &partition(X_0, \dots, X_{r-1}) \land \; \forall y \; \left(  (x_{\mathit{first}} \le_1 y) \land (y \le_1 x_{\mathit{last}}) \right) \\
				\land \; &\bigwedge_{a = 0}^{r-1} \left(\psi^{\ell + 1}_a(\mathbf{x} \circ (x_{\mathit{first}})) \rightarrow x_{\mathit{first}} \in X_a \; \right)\\
				\land \; & \forall y \forall z \left( S_{<_1}(y,z) \rightarrow \bigwedge_{a = 0}^{r-1} \bigwedge_{b = 0}^{r-1}
				\left( \begin{gathered} \left( y \in X_a \land \psi^{\ell + 1}_b(\mathbf{x} \circ (z)) \right) \\  \rightarrow \;  z \in X_{(a+b) \bmod r}\end{gathered} \right)  \right)\\
				\land \; & x_{\mathit{last}} \in X_s 
			\end{aligned}
		\right)
	\end{multline*}
	where $S_{<_1}(y,z)$ defines the successor relation induced by the order $<_1$, i.e., it verifies that $z$ is the leftmost element to the right of $y$.  
	
	Suppose that we have a permutation $\pi$ of size $n$ and we are evaluating $\psi^\ell_s(\bfx)$ on~$\pi$ with its free variables bound to an $\ell$-tuple $\bfa \in \pi^\ell$.
	Let us explain the semantics of the partition $X_0, \dots, X_{r-1}$.
	An element $p$ lies in $X_t$ if and only if the number of $(k-\ell)$-tuples that (i) complete the tuple $\bfa$ to a $k$-tuple satisfying~$\varphi$, and (ii) start with an element $q$ such that $q \preceq_1 p$, is congruent to~$t$ modulo~$r$.

	Let $p_1, \dots, p_n$ be the elements of $\pi$ ordered from left to right, i.e., according to the order $\prec_1$.
	Inductively, we define for each $j \in [n]$ an index $i_j \in \{0, \dots, r-1\}$ as follows
	\begin{align}
		i_1 &\equiv \left|\{ \bfb \mid \bfb \in \pi^{k - \ell - 1} \land (\pi, \bfa \circ (p_1) \circ \bfb)\} \right| &&\pmod{r}\label{eq:idef-1}\\
		i_{j} \; &\equiv \;  i_{j-1} + \left|\{ \bfb \mid \bfb \in \pi^{k - \ell - 1} \land (\pi, \bfa \circ (p_j) \circ \bfb)\}\right| &&\pmod{r}\label{eq:idef-j}
	\end{align}
	
	Observe that we can rephrase the task as verifying that $i_n = s$ since $i_n$ is congruent to the number of all the $k-\ell$ tuples that complete $\bfa$ to a satisfying assignment.

	In one direction, the correctness of the definition of $\psi^\ell_s$ is fairly straightforward.
	We state it as an observation as it suffices to check the evaluation while applying induction hypothesis on every occurrence of $\psi^{\ell+1}_s$.

	\begin{observation}
		Suppose that the number of ways to complete the tuple $\bfa$ to a $k$-tuple is congruent to $s$ modulo $r$.
		Then $(\pi,\bfa) \models \psi^\ell_s$ by setting $x_{\mathit{first}} = p_1$, $x_{\mathit{last}} = p_n$ and $X_t = \{p_j \mid j \in [n] \land i_j = t\}$ for each $t \in \{0, \dots, r-1\}$.
	\end{observation}
	
	The other direction is implied by the following claim which guarantees that $i_n = s$ whenever $(\pi,\bfa) \models \psi^\ell_s$.

	\begin{claim}
		Let $\pi$ be an arbitrary permutation and $\bfa \in \pi^\ell$ a tuple such that $(\pi,\bfa) \models \psi^\ell_s$.
		Then we have that $p_j \in X_{i_j}$ for every $j \in [n]$ and moreover, $p_n \in X_s$.
	\end{claim}
	\begin{proofclaim}
	The second line in the definition of $\psi^\ell_s$ guarantees that $x_{\mathit{first}}$ and $x_{\mathit{last}}$ are bound to the elements $p_1$ and $p_n$, respectively.
The third line enforces that $p_1 \in X_{i_1}$ by its definition in~\eqref{eq:idef-1}.
The fourth line models the equation~\eqref{eq:idef-j} and thus, inductively enforces that $p_j \in X_{i_j}$ for each $j \in [n]$.
Finally, it is directly checked on the fifth line that $p_n \in X_s$.
	\end{proofclaim}

	It remains to bound the size of the constructed formulas.
	Observe that $\psi^\ell_s$ is constructed from $r^2 + r$ copies of $\psi^{\ell+1}_a$ with extra $O(r^2)$ symbols, and the length of $\psi^k_s$ is at most $|\varphi|+1$.
	It follows that $|\psi^\ell_s| = r^{O(k-\ell + 1)} \cdot |\varphi|$.
\end{proof}

Let us remark that Proposition~\ref{pro:tuple-cmso} is not surprising and closely related statements can be found in the literature.
That the predicates $card_{q,r}(X)$ become MSO-expressible once a linear order is available is folklore, and it is obtained by essentially the argument above; Ganzow and Rubin~\cite{Ganzow2008} call this translation straightforward.
Modulo-counting atoms applied to relations of arbitrary arity, rather than only to sets, are moreover a standard part of the definition of CMSO used, e.g., by Elberfeld et al.~\cite{Elberfeld2016}.
Finally, in the special case of words, Proposition~\ref{pro:tuple-cmso} also follows from the B\"uchi--Elgot--Trakhtenbrot theorem (Theorem~\ref{thm:mso-regular}), since the language of words in which the number of $k$-tuples satisfying $\varphi$ is congruent to $q$ modulo $r$ is regular.
We nevertheless give a self-contained proof, as we are not aware of a reference stating the result in this exact form, and in particular one providing the bound on the length of $card^\varphi_{q,r}$.

Next, we provide one example of how to verify a modular constraint on a permutation statistic by an MSO sentence in \TOTO using Proposition~\ref{pro:tuple-cmso}.
This has an interesting corollary with regards to enumeration because it implies that permutations with such property within certain well-structured permutation classes are counted by a rational generating function that is moreover algorithmically computable.

\begin{example}
	We will construct an MSO sentence in \TOTO expressing modular constraint on the value of a well-known permutation statistic, the major index.
	For a permutation $\pi$ of size $n$, the \emph{descent set} of~$\pi$, denoted by $\Des(\pi)$, is the set of all $i \in [n-1]$ such that $\pi_i > \pi_{i+1}$.
	The \emph{major index} of $\pi$, denoted by $\maj(\pi)$, is defined as the sum of its descent set, i.e. $\maj(\pi) = \sum_{i \in \Des(\pi)} i$.

	The crucial ingredient is that the major index can be alternatively computed as the total number of occurrences of constantly many generalized patterns.
	A \emph{vincular pattern} is an extension of classical permutation pattern where we addionally require that certain elements are consecutive.
	This is denoted by drawing line in the one-line notation of permutations above all pairs of elements that are required to be consecutive, e.g., the vincular pattern $1\overline{32}$ is contained in a permutation $\pi$ if $\pi$ contains an occurrence of $132$ where the images of $3$ and $2$ are consecutive with respect to $<_1$. 
	Babson and Steingr\'{i}msson~\cite{Babson2000} observed that major index can be computed from the number of occurrences of vincular patterns in the following way
	\[\maj(\pi) = \occ(1\overline{32}, \pi) + \occ(2\overline{31},\pi) + \occ(3\overline{21},\pi) +\occ(\overline{21},\pi)\]
	where $\occ(\sigma, \pi)$ denotes the number of occurrences of $\sigma$ in $\pi$.

	For any vincular pattern $\sigma$ of size $k$, we can easily construct an FO formula~$\varphi_\sigma$ with $k$ free variables such that a permutation $\pi$ with its $k$-tuple of elements $\bfa$ models $\varphi_\sigma$ if and only if $\bfa$ forms an occurrence of $\sigma$.
	See~\cite{Albert2020} for more details.
	With the use of Proposition~\ref{pro:tuple-cmso}, this allows us to construct an MSO sentence describing all permutations such that their value of major index is divisible by three as follows
	\[\bigvee_{\mathclap{\substack{a,b,c,d \, \in\, \{0,1,2\}\\ a+b+c+d \,\equiv\, 0 \pmod{3}}}} \quad card^{\varphi_{1\overline{32}}}_{a,3} \land card^{\varphi_{2\overline{31}}}_{b,3} \land card^{\varphi_{3\overline{21}}}_{c,3} \land card^{\varphi_{\overline{21}}}_{d, 3}.\]
	
	As a consequence, we can apply the results of Braunfeld~\cite[Theorem 1.1]{Braunfeld2023} to show that the number of permutations with major index divisible by three restricted to any geometric grid class of permutations is enumerated by a rational generating function that can moreover be computed algorithmically.
\end{example}

Babson and Steingr\'{i}msson~\cite{Babson2000} give other examples of permutation statistics expressed as counts of generalized patterns and for each such statistic, we can express modular constraints in MSO sentences in \TOTO.
Moreover, Br\"{a}nd\'{e}n and Claesson~\cite{Branden2011} show that many other permutation statistics can be expressed using the counts of more generalized patterns, called mesh patterns.
These include, e.g., the number of left-to-right maxima or the number of sum-indecomposable components.
The occurrence of a fixed mesh pattern is also easily expressed using an FO formula (see~\cite{Albert2020}) and thus, we can again use Proposition~\ref{pro:tuple-cmso} to encode modularity constraints on all these permutation statistics as well.

\subsection{Merge Classes}

Now let us turn our attention to a more structural kind of properties expressible by MSO logic. We begin with a simple 
example.

\begin{example}
\label{ex:merges}
  We can define a predicate $partition(X,Y)$ enforcing that $X$ and $Y$ form a partition of the domain, and predicates $increasing(X)$ and $decreasing(X)$ enforcing that $X$ is an increasing, respectively decreasing point set as follows
  \begin{align*}
    partition(X,Y) &= \forall x\, \left[(x \in X \lor x \in Y) \land \neg(x \in X \land x \in Y)\right],\\
    increasing(X) &= \forall x, y \left[(x \in X \land y \in X) \rightarrow  (x <_1 y \leftrightarrow x <_2 y)\right],\\
    decreasing(X) &= \forall x, y \left[(x \in X \land y \in X) \rightarrow  (x <_1 y \leftrightarrow y <_2 x)\right].
  \end{align*}
  Using these predicates, we can define permutations obtained as a union of an increasing and a decreasing sequence (known as \emph{skew-merged permutations}) by the MSO sentence
  $\exists X \exists Y  \left(partition(X,Y) \land increasing(X) \land decreasing(Y)\right)$.
\end{example}

We may generalize this example to permutations that can be obtained as a union of 
$k$ permutations, each coming from an arbitrary set defined by a fixed MSO sentence.
Formally, a permutation $\pi$ is a \emph{merge} of permutations $\sigma$ and $\tau$ if we can color the points of $\pi$ with colors red and blue so that the red points are isomorphic to $\sigma$ and the blue ones to~$\tau$.
The \emph{merge} of a class $\cC$ and a class $\cD$ is the class $\cC\odot\cD$ of permutations that can be obtained by merging an element of $\cC$ with an element of~$\cD$.
The following proposition is a straightforward generalization of Example~\ref{ex:merges}.

\begin{proposition}
  \label{pro:mso-merges}
  For arbitrary MSO sentences $\varphi_1, \ldots, \varphi_k$ in \TOTO, we can construct an MSO sentence $\rho$ such that $\pi \models \rho$ if and only if $\pi$ can be obtained as a merge of permutations $\pi_1, \ldots, \pi_k$ such that $\pi_i \models \varphi_i$ for every $i \in [k]$.
\end{proposition}
\begin{proof}
  For an MSO sentence $\varphi$, observe that we can define an MSO predicate $\varphi(X)$ that limits the domain of all variables in $\varphi$ to the set $X$ and thus, effectively tests if the set $X$ satisfies $\varphi$.
Technically, this is done by replacing
\begin{align*}
	\exists x \, \psi \, &\longrightarrow\,  \exists x \, ( x \in X \land \psi)  \\
	\forall x \, \psi \, &\longrightarrow\,  \forall x \, ( x \in X \rightarrow \psi)  \\
	\exists Y \, \psi \, &\longrightarrow\,  \exists Y \, ( \forall x\, (x \in Y \rightarrow x \in X) \land \psi) \\
	\forall\, Y \, \psi \, &\longrightarrow\,  \forall Y \, ( \forall x\, (x \in Y \rightarrow x \in X) \rightarrow \psi).
\end{align*}
The desired sentence is then obtained easily using the predicate for testing whether a $k$-tuple of sets $X_1, \ldots, X_k$ forms a partition of the domain as
\[\rho = \exists X_1 \exists X_2\, \cdots\, \exists X_k \left(partition(X_1, \ldots, X_k) \land \bigwedge_{i=1}^k \varphi_i(X_i)\right).\qedhere\]
\end{proof}

The notion of merges is not merely an artificial example of an MSO-definable 
property; in fact, the concept of merging has recently attracted a fair amount of 
attention.
It has been originally introduced as an approach for the
enumeration of pattern-avoiding permutations~\cite{Albert07,Albert2019}.
Most notably, Claesson et al.~\cite{Claesson2012} have shown that every 1324-avoiding permutation can
be obtained as a merge of a 132-avoiding permutation with a 213-avoiding one, and
this result, together with its subsequent strengthenings by Bóna~\cite{Bona2014} and
Bevan et al.~\cite{Bevan2017}, are the basis of the best known upper bounds for the
number of 1324-avoiding permutations.
Apart from enumeration questions, the research on permutation merges has also focused on structural 
issues, such as whether a given permutation class can be
obtained by merging two of its proper subclasses~\cite{Jelinek2015,Jelinek2017}, or which
classes can be obtained by merging a bounded number of permutations from a given
class~\cite{Albert2016a,Vatter2016}.

In contrast with Proposition~\ref{pro:mso-merges}, we will show that some merges cannot be 
expressed by an FO sentence.
Recall that a permutation is simple if it cannot be obtained as an inflation of a strictly smaller permutation.
We will show that for any simple permutation $\alpha$ of length at least 4, there is no FO 
sentence expressing that a permutation is a merge of two $\alpha$-avoiding permutations.
To this end, we first introduce a standard tool for proving inexpressibility in logic -- the 
Ehrenfeucht--Fraïssé games.

\paragraph*{Ehrenfeucht--Fraïssé games.}
Let $(A, I)$ and $(B, J)$ be two models of the same theory, and let $k$ be a positive integer.
The \emph{$k$-move Ehrenfeucht--Fraïssé (EF) game} is a game between two players, called Spoiler 
and Duplicator, on the models $(A, I)$ and $(B, J)$ with the following rules.
Spoiler begins and the players alternate in moves until they both had exactly $k$ turns.
In the $i$-th turn, Spoiler chooses either an element $a_i \in A$ or $b_i \in B$ and Duplicator replies by
choosing an element of the other model; that is, either Spoiler chooses an element $a_i \in A$ and
Duplicator responds by choosing $b_i \in B$, or Spoiler chooses an element $b_i \in B$ and Duplicator
responds by choosing $a_i \in A$.
At the end of the game, Duplicator wins if the map $a_i \mapsto b_i$ is an isomorphism between the submodels
induced by $\{a_i \mid i \in [k]\}$ and $\{b_i \mid i \in [k]\}$.
Otherwise, Spoiler wins.

We assume that both players play optimally and we say that Duplicator \emph{wins} a given EF game, if he has a winning strategy.
We denote by $(A, I) \sim_k (B, J)$ that Duplicator wins in the $k$-move EF game on $(A, I)$ and $(B, J)$.

Before establishing the connection of EF games to FO logic, we need one more definition.
The \emph{quantifier depth of an FO formula $\varphi$}, denoted by $\qd(\varphi)$, is defined recursively as (i) $\qd(\varphi) = 0$ for any atomic formula $\varphi$, (ii) $\qd(\neg \varphi) = \qd(\varphi)$, (iii) $\qd(\varphi\, \square\, \psi) = \max(\qd(\varphi), \qd(\psi))$ for any binary operation $\square$, and (iv) $\qd(\exists x \varphi) = \qd(\forall x \varphi) = \qd(\varphi) + 1$.


\begin{theorem}[\cite{Fraisse1954,Ehrenfeucht1960}]
  For two models $(A,I)$ and $(B,J)$ of the same theory, we have $(A, I) \sim_k (B, J)$ if and only if $(A,I)$ and $(B,J)$ satisfy the same set of sentences of quantifier depth at most $k$.
\end{theorem}

Let us provide a simple example of EF games on linear orders.
Formally, the \emph{Theory of Linear Orders} (\TOLO) is defined on the signature with a single binary relation symbol $<$ where the axioms enforce that $<$ is a linear order.
For simplicity, we write the models of \TOLO{} as pairs $(A, \prec)$ where $\prec$ is a linear order on the domain $A$.
It is well-known that we cannot distinguish two sufficiently large linear orders with FO sentences of a fixed quantifier depth.

\begin{proposition}[{\cite{Gradel2007}}]
  \label{pro:ef-linear-orders}
  Let $k$ be a positive integer. If $(A, \prec_A)$ and $(B, \prec_B)$ are finite models of \TOLO{} such that $|A|, |B| \geq 2^k - 1$, then $(A, \prec_A) \sim_k (B, \prec_B)$ and thus, $(A, \prec_A)$ and $(B, \prec_B)$ satisfy the same set of sentences of quantifier depth at most $k$.
\end{proposition}

Now we possess all the tools necessary to prove that FO sentences in \TOTO{} are not powerful enough 
to express the property that a permutation can be merged from two smaller permutations 
avoiding a fixed simple pattern.
Note that such property is MSO-definable by Proposition~\ref{pro:mso-merges}.

\begin{theorem}
  \label{thm:fo-merges}
  Let $\alpha$ be a simple permutation of length at least 4.
  The class $\Av(\alpha) \odot \Av(\alpha)$ is not definable by an FO sentence in \TOTO.
\end{theorem}
\begin{proof}
Let us start with some observations and assumptions.
First, let $m$ be the length of $\alpha$ and observe that $\alpha$ must contain either $2413$ or $3142$, because otherwise it would be separable and there are no separable simple permutations of length larger than~3.
Furthermore, we can assume without loss of generality that $\alpha$ contains $3142$, as otherwise its reverse $\alpha^r$ would contain $3142$ and clearly $\Av(\alpha) \odot \Av(\alpha)$ is definable by an FO sentence if and only if $\Av(\alpha^r) \odot \Av(\alpha^r)$ is.
Additionally, the class $\Av(\alpha)$ is closed under inflations, i.e. for any $\tau \in \Av(\alpha)$ of length $k$ and any $k$ permutations $\sigma_1, \dots, \sigma_k \in \Av(\alpha)$ the inflation $\tau[\sigma_1, \dots, \sigma_k]$ is also $\alpha$-avoiding.

We define $\alpha^\ra$, $\alpha^\la$, $\alpha^\ta$ and $\alpha^\da$ to be the permutations obtained from $\alpha$ by removing the rightmost, leftmost, topmost and bottommost element, respectively.
We say that a point set $P$ forms a \emph{right arrow} if it is isomorphic to $\alpha^\ra$.
Furthermore, we say that a point $p$ is in the range of the right arrow $P$, if $p$ lies to the right of $P$ and the point set $P \cup \{p\}$ is isomorphic to $\alpha$.
Similarly, we define \emph{top, left and down arrows} as point sets isomorphic to $\alpha^\ta, \alpha^\la$ and $\alpha^\da$, respectively.
Their ranges are defined analogously.

An \emph{admissible coloring} of a permutation $\pi$ is a 2-coloring $\psi\colon \pi \to\{\text{red}, \text{blue}\}$ such that $\pi$ does not contain a monochromatic copy of $\alpha$.
Assume we have a permutation $\pi$ with an admissible coloring.
Observe that if $\pi$ contains a monochromatic arrow of any orientation, say red, then all the points in the range of the arrow must be colored by the other color, i.e. blue.

Our goal is to construct, for each positive $k$, two permutations that are indistinguishable by FO sentences of quantifier depth $k$ and simultaneously, only one of them belongs to the class $\Av(\alpha) \odot \Av(\alpha)$.
To that end, we define a permutation~$\pi_\ell$ for each positive integer $\ell$.
We build $\pi_\ell$ from $4 \ell + 2$ blocks $B_1, B_2, \dots, B_{4\ell + 2}$ forming a clockwise spiral starting with the innermost block~$B_1$.
Formally, if $i$ is a multiple of four then any element in the block $B_i$ is to the left of all the elements in the block $B_{i-1}$ , and it is to the left and below all the elements in the blocks $B_1, \dots, B_{i-2}$.
For $i \equiv 1 \pmod{4}$, the block $B_i$ lies above $B_{i-1}$ and to the left and above the blocks $B_1, \dots, B_{i-2}$.
For $i \equiv 2 \pmod{4}$, the block $B_i$ lies to the right of $B_{i-1}$ and to the right and above the blocks $B_1, \dots, B_{i-2}$.
For $i \equiv 3 \pmod{4}$, the block $B_i$ lies below $B_{i-1}$ and to the right and below the blocks $B_1, \dots, B_{i-2}$.
Refer to Figure~\ref{fig:merges-ef}.

\begin{figure}
  \centering
  \includegraphics{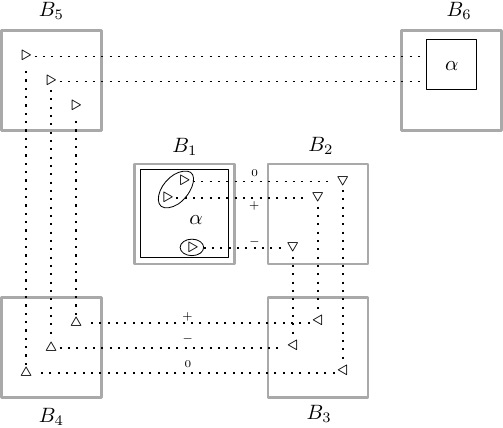}
  \caption{Construction of the permutation $\pi_\ell$ for $\ell = 1$ in the proof of Theorem~\ref{thm:fo-merges} that contains three tracks of arrows -- ground ($0$), positive ($+$) and negative ($-$).   }
  \label{fig:merges-ef}
\end{figure}

The first block $B_1$ consists of the permutation $\alpha$ with the topmost point inflated by $\alpha^\ra\oplus\alpha^\ra$, the bottommost point inflated by $\alpha^\ra$, and every other point inflated with $\alpha$.
We  will call the inflated topmost and bottommost elements of $\alpha$ the \emph{top chunk} and \emph{bottom chunk}, respectively.
Observe that in any admissible coloring of $B_1$, both the top chunk and the bottom 
chunk must be monochromatic and moreover, they must use different colors.
This is because all the other inflated points of $\alpha$ necessarily contain elements of both colors.
Therefore, a pair of elements in the top and bottom chunks sharing the same color would create a monochromatic copy of $\alpha$.
The last block $B_{4\ell+2}$ contains only the permutation~$\alpha$.

Let us now describe the intermediate blocks $B_2, \dots, B_{4\ell+1}$.
Recall that the first block $B_1$ contains two right arrows in its top chunk and one right arrow in its bottom chunk.
The block $B_2$ will contain 3 down arrows, each down arrow placed fully in the range of one of the 3 right arrows in $B_1$.
Similarly, the block~$B_3$ will contain 3 left arrows, each in the range of a distinct down arrow from $B_2$.
Generally, the block $B_i$ for any $i \in \{2,\dots, 4\ell+2\}$ will consist of 3 disjoint arrows, all oriented towards $B_{i+1}$, and each of them inside the range of a distinct arrow in $B_{i-1}$.

We continue by specifying the relative position of the arrows inside each block~$B_i$.
For every even $i$, the arrows inside the block $B_i$ form an increasing sequence.
In other words, $B_i$ is isomorphic to a direct sum of three arrows.
For odd $i$, we distinguish two cases.
If $i = 4t + 1$ for some $t$, the right arrows inside the block $B_i$ form a decreasing sequence.
Finally if $i = 4t + 3$ for some $t$, the block $B_i$ is isomorphic to $231$ inflated with three left arrows, i.e. $231[\alpha^\la,\alpha^\la,\alpha^\la]$.
We say that $B_i$ is a \emph{monotone block} whenever the arrows in $B_i$ form a monotone sequence.
Moreover, we distinguish the arrows in monotone blocks based on the distance from the center of the spiral as \emph{inner}, \emph{middle} and \emph{outer}.
In the case when $i = 4t+3$, it makes no sense to define inner and middle arrows.
However, we say that the arrow inflating the element `1' of $231$ is also an outer arrow.

Notice that the arrows form three disjoint sequences such that each sequence contains exactly one arrow in each of the blocks $B_1, \dots, B_{4\ell + 1}$ and moreover, the arrow in the block $B_i$ for $i \ge 2$ lies in the range of the arrow in the block $B_{i-1}$.
We call these sequences of arrows \emph{tracks}.
In particular, one track contains all the outer arrows and the remaining two switch between middle and inner arrows whenever they pass through a block $B_{4t+3}$ for some~$t$.
The track containing all the outer arrows is called the \emph{ground}, the other track containing an arrow from the top chunk of $B_1$ is called \emph{positive} and the track starting with the arrow in the bottom chunk of $B_1$ is called \emph{negative}.

Finally, it remains to describe the relative positions between the points in the blocks $B_{4\ell + 1}$ and $B_{4\ell+2}$.
We place the topmost point in $B_{4\ell + 2}$ in the range of the outer right arrow in $B_{4\ell + 1}$ and we place the remaining points in the range of the middle right arrow in $B_{4\ell + 1}$.
See Figure~\ref{fig:merges-ef}.

\begin{claim}\label{claim:piell-parity}
	The permutation $\pi_\ell$ belongs to $\Av(\alpha) \odot \Av(\alpha)$ if and only if $\ell$ is odd.
\end{claim}

\begin{proofclaim}
Assume that there is an admissible coloring of $\pi_\ell$.
We have already noticed that the top and bottom chunk in the block $B_1$ are both monochromatic and colored by opposite colors.
Without loss of generality, we assume that the top chunk receives the color red.
As we also observed, the colors of arrows on a single track must alternate and therefore, the rest of the admissible coloring of $\pi_\ell$ is uniquely determined.
In particular, the outer arrow in the block $B_i$ is colored red if and only if $i$ is odd.
The same holds for the arrows contained in the positive track while the arrow in $B_i$ contained in the negative track is red if and only if $i$ is even.

Observe that the middle arrow of a monotone block $B_i$ belongs to the positive track if and only if $\lfloor i/4\rfloor$ is even.
This follows since for every $t$, the block $B_{4t+3}$ flips the positive and negative track, and the middle arrow of the block $B_2$ belongs to the positive track.
In particular, the middle arrow of the block $B_{4\ell + 1}$ belongs to the positive track if and only if $\lfloor (4\ell + 1)/4\rfloor = \ell$ is even.

Let us inspect the case when $\ell$ is even.
The outer arrow of $B_{4\ell + 1}$ is red since it belongs to the ground and so is the middle arrow since it belongs to the positive track.
But then all the elements in the block $B_{4\ell + 2}$ must be colored blue since they all lie in the range of one of these two arrows and in particular, we found a blue copy of $\alpha$.

It remains to show that for odd $\ell$, we have an admissible coloring.
To see this, let $\pi^R_\ell$ and $\pi^B_\ell$  be the subpermutations of $\pi_\ell$ formed by the red elements and the blue elements, respectively.
We claim that both these permutations avoid~$\alpha$.
Let us look at $\pi^R_\ell$, the case of $\pi^B_\ell$ being analogous.
To check that $\pi^R_\ell$ avoids $\alpha$, we will repeatedly apply the following observation.

\begin{observation}\label{obs-deflate} Suppose that $\gamma$ is a permutation which contains an interval~$I$, and suppose that $I$ has no copy of the  pattern~$\alpha$.
	Let $\gamma^-$ be a permutation obtained from $\gamma$  by `deflating' the interval $I$, i.e., by replacing $I$ by a single element.
	Then $\gamma$ contains $\alpha$ if and only if $\gamma^-$ contains~$\alpha$.
\end{observation}

Our goal is to show that by repeatedly deflating $\alpha$-avoiding intervals, we can 
transform $\pi^R_\ell$ into a simpler permutation which contains only monotone sequence within each block.
We can then apply known fact about such permutations to conclude that it avoids~$\alpha$.  

Consider first the block $B_1$.
The red subset of each inflated point of the copy of $\alpha$ is an interval.
This holds vacuously for all points other than the topmost and the bottommost.
As we already observed, the whole top chunk is red and the bottom chunk is blue or vice versa.
Let us assume that the top chunk is colored red as the other case is symmetric.
Since all the elements in $B_2$ in the range of the two arrows in the top chunk are blue, the top chunk indeed forms an interval in $\pi^R_\ell$.
Thus, the red part of $B_1$ forms in fact a single interval isomorphic to an inflation of $\alpha^\da$, which avoids $\alpha$ and can therefore be deflated to a single point.

Consider now a block $B_i$ for $i \in \{2, \dots, 4\ell + 1\}$.
Notice that it contains at most two arrows colored red.
The elements in the range of any red arrow are colored blue and therefore, each such arrow can be deflated to a single point.
We claim that these points form a monotone pair that is increasing for even~$i$ and decreasing for odd~$i$.
This is immediate in all cases except possibly when $i = 4t+3$ for some $t$.
But in such case, we recall that the arrows of the positive and negative track receive opposite colors within every block and therefore, the two red arrows must form a decreasing pair.

It remains to deal with the last block $B_{4\ell+2}$.
Since $\ell$ is odd, the outer and middle arrow in the block $B_{4\ell+1}$ are colored differently.
If the outer arrow in $B_{4\ell+1}$ is colored blue then there is only a single red point in $B_{4\ell+2}$.
Otherwise, the red points in $B_{4\ell+2}$ are exactly the ones lying in the 
range of the middle arrow in $B_{4\ell+1}$ and thus, they form a copy of $\alpha^\ta$ and they can be 
again safely deflated to a single point.

Through this process, we obtained a permutation that contains an increasing sequence in every block~$B_i$ for even $i$, and a decreasing sequence in every block~$B_i$ for odd~$i$.
Permutation with such spiral decomposition is known to avoid $3142$~\cite[Lemma 4.1]{Jelinek2017a} and thus, it also avoids~$\alpha$.
\end{proofclaim}

\begin{claim}\label{claim:piell-ef}
	Let $k$ be a positive integer.
	For every $n, m \ge 2^{k+1}-2$, we have $\pi_n \sim_k \pi_m$ and thus, 
	$\pi_n$ and $\pi_m$ satisfy the same set of FO sentences of quantifier depth at most~$k$ in \TOTO.
\end{claim}
\begin{proofclaim}
We shall define a strategy for the $k$-move EF game on $\pi_n$ and $\pi_m$ using as a building block the strategy for linear orders.
In particular, let $A = \{0, \dots ,n\}$ and $B = \{0, \dots ,m\}$ be sets of numbers equipped with the natural linear order.
There exists a winning strategy for Duplicator in the $(k+1)$-move EF game on $A$ and $B$ due to Proposition~\ref{pro:ef-linear-orders}.
Let us observe a few of its properties.
If any two elements $a_i$ and $a_j$ picked in $A$ during the first $k$ rounds are successive, i.e., $a_j = a_i + 1$, then necessarily, so are $b_i$ and $b_j$.
Otherwise, Spoiler could win in the last round by selecting $b_{k+1}$ such that $b_i < b_{k+1} < b_j$.
For the same reason, an element $a_i$ for $i \in [k]$ is the minimum element of $A$ if and only if $b_i$ is the minimum element of $B$; the same holds when we replace minimum with maximum.

We are now ready to define the strategy for Duplicator.
At the same time while playing on $\pi_n$ and $\pi_m$, we simulate a virtual $(k+1)$-move EF game on $A$ and $B$.
Suppose Spoiler picks in the $i$-th round an element $p_i$ from the block $B_{c_i}$ in the permutation $\pi_n$ (the other case being symmetric).
We set its $i$-th move in the virtual game as picking the element $a_i = \lfloor \frac{c_i}{4} \rfloor$.
Let $b_i$ be the response of Duplicator in the virtual game and set $d_i = 4 b_i + (c_i \bmod 4)$.
Duplicator picks a point $q_i$ in the permutation $\pi_m$ from the block $B_{d_i}$.
As we discussed, we have $a_i = 1$ if and only if $b_i = 1$ and $a_i = n+1$ if and only if $b_i = m+1$. 
Since moreover $c_i$ and $d_i$ share the same remainder modulo $4$, the blocks $B_{c_i}$ and $B_{d_i}$ are isomorphic as point sets and Duplicator can pick $q_i$ inside $B_{d_i}$ mimicking the choice of $p_i$ in $B_{c_i}$.

We claim that the map $p_i \mapsto q_i$ is an isomorphism between the respective subpermutations of $\pi_n$ and $\pi_m$.
It is sufficient to show that the pair $p_i, p_j$ is isomorphic to $q_i, q_j$ for all choices of $i, j \in [k]$.
That is trivial if $p_i$ and $p_j$ belong to the same block inside $\pi_n$ since then $a_i = a_j$ and thus, also $b_i = b_j$ in the virtual game.
Now let us assume that $p_i$ and $p_j$ lie in different non-successive blocks.
Observe that $a_i \le a_j$ implies $b_i \le b_j$ and moreover, we have $c_i \equiv d_i \pmod 4$ by construction.
It follows that the relative position of $B_{c_i}$ to $B_{c_j}$ is the same as the relative position of $B_{d_i}$ to $B_{d_j}$ and thus, the pair $p_i, p_j$ is isomorphic to $q_i, q_j$.
Finally, suppose that $p_i$ and $p_j$ occupy two successive blocks $B_t$ and $B_{t+1}$, respectively.
Then $q_i$ and $q_j$ must also occupy two successive blocks $B_s$ and $B_{s+1}$, and $q_i, q_j$ is isomorphic to $p_i, p_j$ since $B_s\cup B_{s+1}$ is isomorphic to $B_t\cup B_{t+1}$.
\end{proofclaim}

For any positive $k$, Claims~\ref{claim:piell-parity} and~\ref{claim:piell-ef} together imply that the class $\Av(\alpha)\odot \Av(\alpha)$ cannot be defined by an FO sentence of quantifier depth at most $k$ in \TOTO.
To see this, consider the permutations $\pi_{2^{k+1} - 1}$ and $\pi_{2^{k+1}}$.
Only one of them belongs to the class $\Av(\alpha)\odot \Av(\alpha)$ while they are indistinguishable by FO sentences of quantifier depth at most~$k$.
\end{proof}

We remark that this fits well into the bigger picture.
Jelínek, Opler and Valtr~\cite{JOV} proved that deciding whether $\pi$ belongs to the class $\Av(\alpha) \odot \Av(\alpha)$ is NP-complete for any simple pattern~$\alpha$.
On the other hand, the definition of the class $\Av(\alpha) \odot \Av(\alpha)$ by an FO sentence would imply a linear time algorithm by the FO model checking of Bonnet et al.~\cite{Bonnet2020} and thus, it would actually prove $P = NP$.
However, Theorem~\ref{thm:fo-merges} makes this result independent of the assumption $P \neq NP$.

Moreover, Murphy~\cite{Murphy2003} asked in his PhD thesis whether all merges of two finitely based classes are themselves finitely based.
Theorem~\ref{thm:fo-merges} answers this question negatively by showing that there are in fact infinitely many merges of two principal classes that are not even FO-definable.

\subsection{Properties Inexpressible in MSO}

As our next contribution, we will prove that MSO is still not powerful enough to express the property of having a fixed point. Instead of using an MSO variant of the EF game, 
we will use the Büchi--Elgot--Trakhtenbrot theorem connecting regular languages with MSO logic.

Let us start with formal definitions of words and languages.
We let $\Sigma^*$ denote the set of all words over an alphabet~$\Sigma$.
A \emph{language} is a subset of~$\Sigma^*$.
Finally, a language $L$ is \emph{regular} if it is the language accepted by a finite automaton.

There is a standard way of defining words as models of a logic theory.
For an alphabet $\Sigma$, the signature $\cS_\Sigma$ consists of one binary relation symbol~$<$ and a unary
relation symbol $P_a$ for every $a \in \Sigma$.
The symbol $<$ is intended to describe the linear order of positions in a word, while the symbols $P_a$ form
a partition of its domain and $P_a$ describes the positions occupied by the letter~$a$.
These conditions are easily described by FO sentences, and therefore form a logic theory that we call
the \emph{Theory of Words over $\Sigma$ ($\TOW_\Sigma$)}.

It turns out that the languages definable by MSO sentences in $\TOW_\Sigma$ are exactly the regular languages.
This is known as the Büchi--Elgot--Trakhtenbrot theorem.

\begin{theorem}[\cite{Buchi1960,Elgot1961,Trakhtenbrot1962}]
  \label{thm:mso-regular}
  A language $L \subseteq \Sigma^*$ is regular if and only if it is definable by an MSO sentence in $\TOW_\Sigma$. 
\end{theorem}

With these tools, we can prove that the property of having a fixed point is inexpressible by an MSO sentence in \TOTO.
More precisely, we show that existence of such sentence would allow us to define a non-regular language by an MSO sentence in $\TOW_\Sigma$.
Note that the proof closely follows a similar argument by Albert et al.~\cite[Proposition 30]{Albert2020}.
\begin{proposition}
  \label{pro:mso-fixed-point}
  The property of having a fixed point is not expressible by an MSO sentence in \TOTO.
\end{proposition}
\begin{proof}
Let us assume for contradiction that there exists an MSO sentence $\varphi$ in \TOTO{} expressing the property that a permutation has a fixed point.
We show how to transform~$\varphi$ into an MSO sentence $\rho$ in $\TOW_\Sigma$ for the two-letter
alphabet $\Sigma = \{a, b\}$ such that $\rho$ defines the language $L = \{ a^n\, b\, a^n \mid n \in
\mathbb{N} \}$. This proves the intended claim since $L$ is clearly non-regular by a standard application
of the pumping lemma.

The main ideas is that the evaluation of $\rho$ on a word of the form $ w = a^k \, b\, a^\ell$ simulates the evaluation of $\varphi$ on the permutation $\pi_{k,\ell} = (\oplus^k 1) \ominus 1 \ominus (\oplus^\ell 1)$.
It is easy to see that $\pi_{k,\ell}$ has a fixed point if and only if $k = \ell$.

In order to transform $\varphi$ into an MSO sentence in $\TOW_\Sigma$, we need to replace all atomic formulas of type $x <_1 y$ and $x <_2 z$.
Let us assume that the variable $x_b$ is set to the position of the only letter $b$ in the word $w$.
We replace $x <_1 y$ simply by $x < y$, effectively mapping the word domain to the permutation from left to right.
The situation is more complicated with $<_2$ since we need to decide the truth value based on the positions relative to $x_b$.
We replace $x <_2 y$ by
\[
	(x < y < x_b) \lor (x_b < x < y) \lor (y < x_b < x) \lor (y < x_b \land x = x_b) \lor (y = x_b \land x_b < x).
\]
The first two disjuncts take care of the case when $x < y$, since then both $x$ and $y$ must either lie before or after the letter $b$.
The remaining three concern the case when $y < x$, in which case either $x$ and $y$ are separated by the letter $b$ or at most one of them is equal to it.
Let $\rho'(x_b)$ be the MSO formula in $\TOW_\Sigma$ with the free variable $x_b$ obtained by these replacements.
We need to additionally check that $w$ contains exactly one occurrence of the letter $b$ and assign its position to the variable~$x_b$.
Putting it all together, we get
\[\rho = \exists x_b\, (P_b(x_b) \land \forall x\, (P_b(x) \rightarrow x = x_b) \land \rho'(x_b) ).\qedhere\]%
\end{proof}

\section{MSO Model Checking}
\label{sec:model-checking}

We shift our attention to the complexity of deciding MSO formulas of \TOTO.
On one hand, we show that there exists an FPT algorithm for this problem parameterized by the tree-width and the length of the formula.
On the other hand, we complement this with a negative result showing that checking MSO sentences on permutations from any class with the so-called poly-time computable long path property is as hard as checking MSO sentences on graphs.


\subsection{Algorithm Parameterized by Tree-width}
\label{ssec:algorithm}


In order to show the tractability of MSO model checking, we prove that every permutation can be interpreted by MSO predicates from its incidence graph equipped with suitable labels.
It is then straightforward to apply standard tools for MSO model checking on graphs and obtain an FPT algorithm parameterized by tree-width.

\begin{theorem}\label{thm:mso-checking-tw}
	Given a permutation $\pi$ of length $n$ and an MSO sentence $\varphi$ in \TOTO, we can decide $\pi \models \varphi$  in time $f(|\varphi|, \tw(\pi)) \cdot n$ for some computable function $f$.
\end{theorem}
\begin{proof}
We first show how a permutation as a model in \TOTO can be interpreted from its labeled incidence graph and only afterwards, we describe the ensuing algorithm for MSO model checking. 

We decorate the incidence graph $G_\pi$ with two types of edge labels, $succ_1$ and $succ_2$, and two types of vertex labels, $min_1$ and $min_2$.
The vertex corresponding to the leftmost point in $\pi$ (the minimum element in the $\prec_1$ order) gets label $min_1$ while the one corresponding to the bottommost point in $\pi$  (the minimal element in the $\prec_2$ order) gets label $min_2$.
Note that a single vertex can receive both labels which happens if and only if $\pi = 1 \oplus \sigma$ for some permutation $\sigma$.
The edges connecting successors in the $\prec_1$ order gets label $succ_1$ while the edges connecting successors in the $\prec_2$ order receive label $succ_2$.
Again, a single edge can receive both colors which happens if and only if $\pi$ is $\oplus$-decomposable, i.e., $\pi = \sigma_1 \oplus \sigma_2$ for some non-empty permutations $\sigma_1$ and $\sigma_2$.
In the rest of the proof, we use $G_\pi$ to denote the incidence graph together with these labels.

Let us now formalize the logic of vertex- and edge-labeled graphs.
Technically, it is rather a logic of the associated incidence structure\footnote{Incidence here refers to the incidence between vertices and edges in arbitrary graph not between points in permutations.}.
Formally, the signature consists of two unary relations $edge$ and $vertex$, a binary relation $Inc$ and finally one unary relation $lab$ for every (edge or vertex) label $lab$.
The relations $edge$ and $vertex$ form a partition of the domain and $Inc$ contains only pairs of the form $(e,v)$ where $e$ is an edge and $v$ is vertex.
Moreover, every edge~$e$ is contained in exactly two pairs $(e,v)$ and $(e,w)$ inside $Inc$.
Clearly all these conditions are FO definable and let us call the corresponding logic theory \emph{Theory of Labeled Graphs (\TOLG)} 

Let us note that this theory and the ensuing MSO logic is usually denoted MSO\textsubscript{2} or MS\textsubscript{2} in the literature (see~\cite{Courcelle2012}). However, we choose  to use this non-standard yet clearly equivalent formalization for consistency with the rest of the paper.

Finally, the interpretation of $\pi$ from its labeled incidence graph follows from the following observation.
For simplicity, we state it only for the order $<_1$ even though clearly the same holds for $<_2$ when we replace the leftmost with the bottommost vertex and $succ_1$ with $succ_2$. 

\begin{observation}\label{obs:mso-encoding-graph}
For any element $x$ of $\pi$, let $P_x$ be the set of edges contained in the unique path in $G_\pi$ between the leftmost vertex (labeled $min_1$) and $x$ that uses only edges with label $succ_1$.
Then for two elements $x$ and $y$, we have $x <_1 y$  if and only if $P_x \subseteq P_y$. 
\end{observation}

It remains to show how to translate Observation~\ref{obs:mso-encoding-graph} into MSO predicates interpreting the orders $<_1$ and $<_2$ from $G_\pi$.
Let us denote by $m_1$ and $m_2$ the unique vertices labeled by $min_1$ and $min_2$ respectively.

Let us consider the order $<_1$ as the other one is symmetrical.
We will handle separately the degenerate case when $x$ equals $m_1$.
Otherwise, let $E$ be a set of edges with labels $succ_1$ such that $m_1$ and $x$ the only vertices incident to precisely one edge from $E$.
Since all edges labeled $succ_1$ induce a path in~$G_\pi$, their arbitrary subset necessarily induces a disjoint set of paths.
However, we require that $m_1$ and $x$ are the only vertices incident to exactly one edge from $E$ and thus, $E$ must in fact be equal to $P_x$.
It is easy to transform this idea into an MSO predicate capturing the relation between a point $x$ and a set of edges on path to the minimal point with respect to the order $<_\alpha$ for $\alpha \in \{1,2\}$ as
\begin{multline*}
path_\alpha(x, E) =
\forall e \;(e \in E \rightarrow succ_\alpha(e))\; \land\\\
\bigl((E = \emptyset \, \land \, min_\alpha(x)) \,\lor\, \forall v\; (\text{`$deg_E(v) = 1$'} \leftrightarrow (v=x \,\lor\, min_\alpha(v))) \bigr).
\end{multline*}
where `$deg_E(v) = 1$' expresses that $v$ is incident to exactly one edge from~$E$ and can be expressed in MSO as
\[\exists f \left( f \in E \land Inc(f,v) \,\land\, \forall g \, \left(\left(g \in E \land Inc(g,v)\right) \rightarrow g=f \right)\right).\]

The final MSO sentence is obtained by replacing every occurrence of $x <_\alpha y$ for $\alpha \in \{1,2\}$ in~$\varphi$ with
\[ \exists E \, \exists F \; \left(path_\alpha(x, E) \land path_\alpha(y, F) \land E \subsetneq F\right).\]
%

The correctness follows straightforwardly from Observation~\ref{obs:mso-encoding-graph} and we get that $\pi \models \varphi$ if and only if $G_\pi \models \rho$.

Finally, let us describe the actual algorithm for MSO model checking.
The algorithm starts by constructing the labeled incidence graph~$G_\pi$ in $O(n)$ time.
This is possible trivially if we have access to both linear orders $\prec_1$ and $\prec_2$ of~$\pi$.
When we receive $\pi$ in the standard one-line notation, i.e. as a sequence of unique values of~$[n]$, we can obtain the order $\prec_2$ in linear time by sorting the input sequence using the standard bucket sort algorithm~\cite{Cormen2009}, which runs in time $O(n+k)$ on $n$ values drawn from a range of $k$ buckets and thus in time $O(n)$ in our setting, since the input consists of $n$ distinct values from~$[n]$.
Afterwards, the algorithm syntactically rewrites $\varphi$ to $\rho$ in $O(|\varphi|)$ time.
Finally, it invokes the standard Courcelle's MSO\textsubscript{2} model checking algorithm on vertex- and edge-labeled graphs of bounded tree-width~\cite[Theorem 6.4]{Courcelle2012} which runs in time $f(|\rho|,\tw(G_\pi)) \cdot n$.
\end{proof}

\subsection{Hardness for Classes with the Long Path Property}


Before we state our hardness result, we need to briefly introduce the concept of monotone grid classes and the (poly-time computable) long path property.

\paragraph*{Monotone grid classes.}
A \emph{monotone gridding matrix of size $k\times\ell$} is a matrix $\cM$ with $k$ columns and $\ell$ rows, whose every entry is one of the three permutation classes $\emptyset$, $\Av(21)$ or $\Av(12)$.
Let $\pi$ be a permutation of length~$n$.
A \emph{$(k\times\ell)$-gridding} of $\pi$ is a pair of non-decreasing sequences $1=c_1\le c_2\le\dotsb\le c_{k+1}=n+1$ and $1=r_1\le r_2\le\dotsb\le r_{\ell+1}=n+1$.
For $i\in[k]$ and $j\in[\ell]$, the \emph{$(i,j)$-cell} of the gridding of $\pi$ is the set of points $p\in S_\pi$ satisfying $c_i\le p.x<c_{i+1}$ and $r_j\le p.y< r_{j+1}$.
A permutation $\pi$ together with a gridding $(c,r)$ forms a \emph{gridded permutation}.
Let $\cM$ be a monotone gridding matrix of size $k\times\ell$.
We say that the gridding of $\pi$ is an \emph{$\cM$-gridding} if for every $i\in[k]$ and $j\in [\ell]$, the subpermutation of $\pi$ induced by the points in the $(i,j)$-cell of the gridding of $\pi$ belongs to the class $\cM_{i,j}$.

We let $\Grid(\cM)$ denote the set of permutations that admit an $\cM$-gridding.
This is clearly a permutation class.
Many important properties of the grid class $\Grid(\cM)$ can be characterized using properties of a certain graph associated to $\cM$.
The \emph{cell graph} of a gridding matrix $\cM$, denoted by $G_\cM$, is the graph whose vertices are all the non-empty entries of $\cM$ and two vertices are adjacent if they appear in the same row or the same column of $\cM$, and there is no other non-empty entry between them.
See Figure~\ref{fig:grid-class}.

\begin{figure}
  \centering
  \raisebox{-0.5\height}{\includegraphics[width=0.45\textwidth]{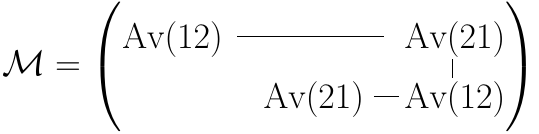}}
  \hspace{0.5in}
  \raisebox{-0.5\height}{\includegraphics[scale=0.5]{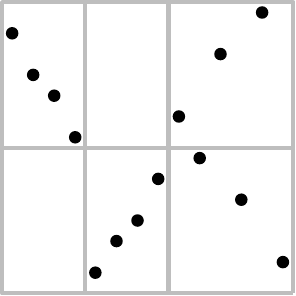}}
  \caption{A monotone gridding matrix $\cM$ on the left and a permutation equipped with an $\cM$-gridding on the right. Empty 
    entries of $\cM$ are omitted and the edges of $G_\cM$ are drawn in $\cM$.}
  \label{fig:grid-class}
\end{figure}

\paragraph*{Long path property.}
We say that a permutation class $\cC$ has the \emph{long path property (LPP)} if for every $k$ the class $\cC$ contains a monotone grid subclass whose cell graph is a path of length~$k$ such that no three consecutive cells share the same row or column.
Moreover, $\cC$ has the \emph{poly-time computable long path property} if there exists an algorithm that given a positive integer $k$ outputs in time polynomial in $k$ a monotone gridding matrix $\cM$ such that $\Grid(\cM) \subseteq \cC$ and the cell graph of $\cM$ is such a path of length~$k$.

There exists a strong connection between the LPP and tree-width.
Jelínek et al.~\cite{Jelinek2021b} showed that any class with the LPP has unbounded tree-width.
In fact, no examples of classes without the LPP having unbounded tree-width are known and we conjecture that these two properties are equivalent.
Moreover, the LPP is not very restrictive as, e.g., the class $\Av(\sigma)$ has the poly-time computable LPP for any $\sigma$ not symmetric to any of $1, 12$ or $132$ (see~\cite{Jelinek2021b}).

\smallskip
We complement Theorem~\ref{thm:mso-checking-tw} by showing that MSO model checking on permutations is 
 as hard as MSO model checking on general graphs, even when we restrict the permutation $\pi$ to a fixed class with the poly-time computable LPP.
Consequently, we show that there exists an explicit MSO sentence~$\varphi$ in \TOTO such that deciding whether a given permutation satisfies~$\varphi$ is NP-hard in any class with the poly-time computable LPP.
We stress that~$\varphi$ is completely independent of the particular class.
If it holds that the LPP is actually equivalent to having unbounded tree-width,  Theorems~\ref{thm:mso-checking-tw} and~\ref{thm:mso-checking-np-hard} form a nice dichotomy (up to the computability assumption) for hardness of MSO model checking inside a fixed permutation class.

\begin{theorem}
	\label{thm:mso-checking-np-hard}
	There exists an MSO sentence~$\psi$ in \TOTO such that deciding whether a permutation $\pi$ satisfies~$\psi$ is NP-hard even when the inputs are restricted to an arbitrary permutation class with the poly-time computable long path property.
\end{theorem}

We need to formalize the logic of undirected graphs before proving the result.
Note that unlike in the proof of Theorem~\ref{thm:mso-checking-tw}, here we are interested in encoding graphs as relational structures where the domain consist only of the vertices.
To that end, we define the signature $\cS_G$ which consists of a single binary relation symbol $E$ describing the edges.
Our only requirement is that $E$ is a symmetric relation, which can be easily described by an FO sentence.
Therefore, we get a theory called \emph{Theory of Graphs ($\TOG$)}.
The MSO logic of this theory is typically denoted MSO\textsubscript{1} or MS\textsubscript{1} in the literature (see~\cite{Courcelle2012}).
Theorem~\ref{thm:mso-checking-np-hard} then follows as a consequence of the following proposition.

\begin{proposition}
  \label{pro:mso-checking-hard}
  Let $\cC$ be a permutation class with the poly-time computable long path property.
  There is a polynomial time algorithm that given a graph $G$ on $n$ vertices and an MSO sentence $\varphi$ in \TOG, computes a permutation $\pi \in \cC$ of length $O(n^2)$ and an MSO sentence $\psi$ in \TOTO{} of length $O(|\varphi|)$ such that $G \models \varphi$ if and only if $\pi \models \psi$.
  Moreover, the sentence $\psi$ depends only on $\varphi$ and in particular, it is independent of the choice of $\cC$.
\end{proposition}

\begin{proof}[Proof of Theorem~\ref{thm:mso-checking-np-hard}]
The famous NP-hard problem \textsc{3-colorability} is easily expressed by an MSO sentence $\varphi$ in \TOG~\cite{Courcelle2012}.
Let $\cC$ be an arbitrary permutation class with the poly-time computable long path property and let $\psi$ be the MSO sentence in \TOTO given by Proposition~\ref{pro:mso-checking-hard}.
Note that $\psi$ depends only on $\varphi$.
Proposition~\ref{pro:mso-checking-hard} then itself serves as a polynomial time reduction between \textsc{3-colorability} and deciding whether $\pi \models \psi$ for a given permutation~$\pi \in \cC$.
\end{proof}

\begin{proof}[Proof of Proposition~\ref{pro:mso-checking-hard}]
We assume that the vertex set of $G$ is precisely the set $[n]$.
The basic idea is that we can represent the adjacency matrix of $G$ by a permutation $\pi \in \cC$. 
We split the adjacency matrix into individual rows and we represent each row using a single cell along a path in the cell graph.

\paragraph*{Construction of $\pi$.}
We first describe the construction of the permutation $\pi$ as it is more straightforward.
We obtain in polynomial time a monotone gridding matrix $\cM$ such that
\begin{enumerate}[label=(\roman*)]
	\item $\Grid(\cM)$ is a subclass of $\cC$,\label{cond:mso-check1}
	\item the cell graph $G_\cM$ is a path on $n+3$ vertices such that no three consecutive cells occupy the same row or column and moreover,\label{cond:mso-check2}
	\item one endpoint of this path is the single non-empty cell in the leftmost column of $\cM$.\label{cond:mso-check3}
\end{enumerate}
The properties~\ref{cond:mso-check1} and~\ref{cond:mso-check2} are directly implied by the poly-time computable long path property.
In order to guarantee~\ref{cond:mso-check3}, we can generate a monotone gridding matrix $\cM'$ that satisfies~\ref{cond:mso-check1} and whose cell graph is a path on $2n + 6$ vertices.
We can then split the path in $G_{\cM'}$ by removing either of the (at most two) non-empty cells in the leftmost column.
One choice leads to a path of length at least $n+3$ with an endpoint in its leftmost column.

\begin{figure}[t!]
	\centering
	\includegraphics[scale=0.8]{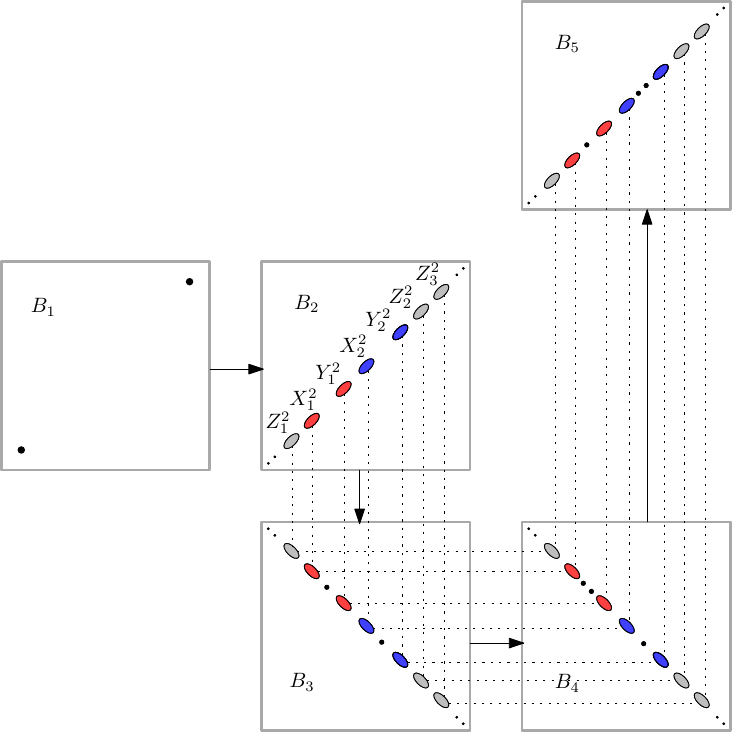}
	\caption{Encoding the complete graph on 2~vertices into a path consisting of 5~monotone blocks in the proof of Proposition~\ref{pro:mso-checking-hard}. Atomic pairs are displayed as small ellipses that are connected into tracks via dotted lines.
		For a fixed $i$, the $X_i$-track and the $Y_i$-track are colored with the same color (red and blue, respectively), while the $Z_1$-track, $Z_2$-track and $Z_3$-track are gray. }
	\label{fig:mso-hard}
\end{figure}

We orient the path in the cell graph outwards from the leftmost cell in $\cM$ and denote the vertices in the order on the path as $v_1, \dots, v_{n+3}$.
We construct $\pi$ as a sequence of $n+3$ monotone blocks $B_1, \dots, B_{n+3}$ such that $B_i$ is increasing if and only if $\cM_{v_i} = \Av(21)$ and for different $i, j \in [n+3]$, the relative position of $B_i$ with respect to $B_j$ is the same as the relative position of $v_i$ with respect to $v_j$.
It then trivially follows that $\pi \in \Grid(\cM) \subseteq \cC$.

We will construct majority of each block $B_i$ from pairs of points, called \emph{atomic pairs}, that we always treat as a single unit.
For an atomic pair $X$, we say that a point $p$ is \emph{in the range of $X$} if $p$ lies either in the horizontal or in the vertical strip bounded by $X$.
Note that this is similar to the arrows and their ranges defined in the proof of Theorem~\ref{thm:fo-merges}.

We postpone the definition of the block $B_1$ for now.
Fix $i \ge 2$.
An increasing block $B_i$ contains 2 points below and to the left of everything else in $B_i$ and 2 points above and to the right of everything else in $B_i$, called \emph{barricades}.
Analogously, the barricades in a decreasing block~$B_i$ consist of 2 points below and to the right of everything else and 2 points above and to the left of everything else. 
The relative position between barricades in neighboring blocks is not important.
Furthermore, $B_i$ contains three special atomic pairs $Z^i_1, Z^i_2$ and $Z^i_3$ and two atomic pairs $X^i_j$ and $Y^i_j$ associated to each vertex $j$ of $G$.
The first block $B_1$ consists of a single atomic pair $S$, called the \emph{anchor}, such that all the atomic pairs in $B_2$ lies in the range of $S$ but the barricades lie outside.

In the block $B_2$, the pairs are ordered from bottom to top as \[Z^2_1, X^2_1, Y^2_1, X^2_2, Y^2_2, \dots, X^2_n, Y^2_n, Z^2_2, Z^2_3.\]
In other words, $Z^2_1$ is the bottommost pair, followed by the pairs $X^2_j, Y^2_j$ ordered by $j$ and at the very top we have pairs $Z^2_2$ and $Z^2_3$.
In the remaining blocks, we position the pairs such that the pair $\cX^i_j$ for $\cX \in \{X,Y,Z\}$ and $i \ge 3$ lies in the range of $\cX^{i-1}_j$. 

Observe that  for each possible choice of $\cX$ and $j$, the pairs $\cX^2_j, \dots, \cX^{n+3}_j$ form a sequence that contains one pair in each block $B_1, \dots, B_{n+3}$ and moreover, the pair in the block $B_i$ for $i \ge3$ lies in the range of the pair in the block $B_{i-1}$.
We again call these sequences \emph{tracks} and in particular, we call the track formed by the pairs $\cX^2_j, \dots, \cX^{n+3}_j, $ the \emph{$\cX_j$-track}.
So, for instance, we have the $Z_1$-track, $X_2$-track, $Y_3$-track etc.

To finish the construction of $\pi$, we add some additional individual points to $\pi$.
We add a single point to the block $B_3$ between the pairs $X^3_j$ and $Y^3_j$ for each $j \in [n]$.
These are used to identify the pairs of tracks associated to a single vertex. 
Each remaining block represents one row of the adjacency matrix of $G$.
For $i \in [n]$, the block $B_{i+3}$ contains additionally a pair of points between the pairs $X^{i+3}_i$ and $Y^{i+3}_i$ and moreover, a single point between $X^{i+3}_j$ and $Y^{i+3}_j$ for any $j \in [n]$ such that $\{i,j\}$ is an edge in $G$.
See Figure~\ref{fig:mso-hard}.

\paragraph*{Construction of $\psi$.}
Now we describe how to translate the MSO sentence $\varphi$ in \TOG{} to the MSO sentence $\psi$ in \TOTO.
Note that we shall define $\psi$ using an expanded language that is however easily translated to MSO sentences.
For two points $p,q$ in $\pi$ and $\alpha \in \{1, 2\}$, we denote by $(p,q)_{\prec_\alpha}$ all the points in $\pi$ that lie in the interval $(p, q)$ with respect to the order $\prec_\alpha$.
Similarly, we denote by $[p,q]_{\prec_\alpha}$ all the points in $\pi$ in the interval $[p, q]$ with respect to the order $\prec_\alpha$.
It is easy to see that predicates like `$|(p,q)_{\prec_1}| = 2$' or `$S \subseteq (p,q)_{\prec_1}$'  are easily expressed via MSO sentences.

First let us describe how we can test whether a set variable $T$ in $\psi$ is equal to a single track.
If we are given an atomic pair $D = (p_1,p_2)$ in the block $B_2$ then we claim that its corresponding track (the $D$-track) can be inductively generated by the following rules
\begin{enumerate}[label=(\roman*)]
	\item $p_1,p_2$ belong to $T$,\label{cond:track1}
	\item for every $q_1,q_2 \in T$ such that $(q_1,q_2)_{\prec_1} = \emptyset$ and  $(q_1,q_2)_{\prec_2} = \{r_1, r_2\}$ for pairwise different $r_1, r_2$, the points $r_1, r_2$ also belong to $T$, and\label{cond:track2}
	\item for every $q_1,q_2 \in T$ such that $(q_1,q_2)_{\prec_2} = \emptyset$ and  $(q_1,q_2)_{\prec_1} = \{r_1, r_2\}$ for pairwise different $r_1, r_2$, the points $r_1, r_2$ also belong to $T$.\label{cond:track3}
\end{enumerate}

\begin{claim}
	\label{claim:track}
	Let $D = \{p_1, p_2\}$ be an atomic pair in the block $B_2$.
	A subset $T$ of $\pi$ is a $D$-track if and only if \ref{cond:track1}, \ref{cond:track2}, \ref{cond:track3} hold for $T$ and moreover, $T$ is minimal such set with respect to inclusion.
\end{claim}
\begin{proofclaim}
Suppose that $T$ is a $D$-track.
The condition~\ref{cond:track1} holds vacuously and thus, the only way $T$ could violate these conditions is if there was a pair of points $q_1, q_2 \in T$ from different blocks for which the condition~\ref{cond:track2} or~\ref{cond:track3} fails.
In that case, both sets $(q_1, q_2)_{\prec_1}$ and $(q_1,q_2)_{\prec_2}$ would contain at most 2 points.
That is, however, not possible for $q_1$ and $q_2$ from different blocks since at least one of these sets must contain four points of the barricades.
Note that this is actually the sole reason why we added barricades during our construction.
For any proper subset $T'$ of $T$, let $i \ge 2$ be the index such that the block $B_i$ is the first where $T'$ and $T$ differ.
If $i = 2$ then $T'$ violates~\ref{cond:track1}, otherwise $T'$ violates~\ref{cond:track2} or~\ref{cond:track3} for the atomic pair in the block $B_{i-1}$.

For the other direction, assume that $T$ is an inclusion-wise minimal set satisfying the conditions \ref{cond:track1}, \ref{cond:track2} and \ref{cond:track3}.
The atomic pair $D$ in the block $B_2$ belongs to $T$ by condition~\ref{cond:track1}.
It then suffices to alternate using conditions~\ref{cond:track2} and~\ref{cond:track3} to show that $T$ must contain the whole $D$-track.
Since the $D$-track itself satisfies \ref{cond:track1}, \ref{cond:track2} and \ref{cond:track3}, the set $T$ cannot contain any other points as it would not be inclusion-wise minimal.
\end{proofclaim}

The conditions~\ref{cond:track1},~\ref{cond:track2} and~\ref{cond:track3} can easily be encoded as an MSO predicate $suptrack(p_1,p_2,T)$.
The power of MSO allows us to further enforce that $T$ is a minimal set satisfying these conditions (and thus, equal to the desired track by Claim~\ref{claim:track}) by an MSO predicate $track(p_1,p_2,T)$ defined as
\[ track(p_1, p_2, T) = suptrack(p_1, p_2, T) \land \forall S \left[(S \subsetneq T) \rightarrow \neg\, suptrack(p_1, p_2, S)\right].\]

At the very beginning of $\psi$, we can fix two variables $a_1$ and $a_2$ to the anchors of $\pi$ since they are the two leftmost points of $\pi$ as follows
\[\exists a_1 \exists a_2 \left[a_1 <_2 a_2 \land \forall x \;(x = a_1 \lor x = a_2 \lor ( a_1 <_1 x \land a_2 <_1 x) ) \right]. \]

We additionally introduce three set variables $T_{Z_1}$, $T_{Z_2}$ and  $T_{Z_3}$, set to the $Z_1$-track, $Z_2$-track and $Z_3$-track, respectively.
For $\alpha \in [2]$, let $p \sqsubset_\alpha q$ be the FO predicate that evaluates true if and only if $p <_\alpha q$ and there is no point $r$ such that $p <_\alpha r <_\alpha q$.
It is sufficient to find the atomic pairs $Z_1$, $Z_2$ and $Z_3$ inside $B_2$ and then apply the predicate $track$.
We can find the atomic pairs using the predicate $\sqsubset_\alpha$ since $Z_1$ contains the bottommost two points in the set $[a_1,a_2]_{\prec_2}$ while $Z_2$ and $Z_3$ consist of the topmost four points in $[a_1,a_2]_{\prec_2}$.
Formally, we write
\begin{multline*}
	\exists\, T_{Z_1}, T_{Z_2}, T_{Z_3} \; [ \exists p_1, q_1, p_2, q_2, p_3, q_3 \;(a_1 \sqsubset_2 p_1 \sqsubset_2 q_1\; \land\; p_2 \sqsubset_2 q_2 \sqsubset_2 p_3 \sqsubset_2 q_3  \sqsubset_2 a_2\\
	\land track(p_1, q_1, T_{Z_1}) \land track(p_2, q_2, T_{Z_2}) \land track(p_3, q_3, T_{Z_3}))].
\end{multline*}

Then we replace every vertex variable $x$ in $\varphi$ with two set variables $T_{X_x}$ and $T_{Y_x}$ intended to describe the $X_x$-track and $Y_x$-track.
With that in mind, we define predicate $vertex(T_{X_x}, T_{Y_x})$ that tests this requirement.
It is sufficient to find two consecutive atomic pairs $\{p_1, q_1\}$ and $\{p_2, q_2\}$ in the range of $\{a_1,a_2\}$ such that moreover, there are two points in the sets $(p_1, q_1)_{\prec_1}$, $(p_2, q_2)_{\prec_1}$ (the atomic pairs $X_x$ and $Y_x$ inside $B_3$) and only one in the set $(q_1,p_2)_{\prec_1}$ (the additional point separating $X_x$ and $Y_x$ inside $B_3$).
Formally, we have
\begin{multline*}
	vertex(T_{X_x}, T_{Y_x}) = \exists\, p_1, q_1, p_2, q_2 \\
	\left(
	\begin{aligned}
	& p_1, q_1, p_2, q_2 \in (a_1, a_2)_{\prec_2} \,\land\,p_1 \sqsubset_2 q_1 \sqsubset_2 p_2 \sqsubset_2 q_2 \\
	\land\,& |(p_1, q_1)_{\prec_1}| = 2
	\,\land\, |(p_2, q_2)_{\prec_1}| = 2 \,\land\, |(q_1, p_2)_{\prec_1}| = 1 \\
	\,\land\,& track(p_1, q_1, T_{X_x}) \,\land\, track(p_2, q_2, T_{Y_x}).
	\end{aligned}
	\right)
\end{multline*}

We replace every occurrence of $\exists x\; \rho$ in $\varphi$ with $\exists T_{X_x}, T_{Y_x} (vertex(T_{X_x}, T_{Y_x}) \land \rho)$, and similarly every $\forall x\; \rho$ is replaced with $\forall T_{X_x}, T_{Y_x} (vertex(T_{X_x}, T_{Y_x}) \rightarrow \rho)$.
Furthermore, any occurrence of $x \in X$ is replaced simply with $T_{X_x} \subseteq X\, \land\, T_{Y_x} \subseteq X$ and we change quantifications over sets to capture only those that are formed as union of tracks $T_{X_x}, T_{Y_x}$ for some set of vertices.
Formally, we replace $\exists X\; \rho$ and $\forall X\; \rho$, respectively, with
\begin{gather*}
	\exists X \; \big[\forall x \in X\; \exists T_{A}, T_{B}\; (vertex(T_{A}, T_{B}) \,\land\, T_{A} \cup T_{B} \subseteq X \,\land\, x \in T_{A} \cup T_{B})\big] \,\land\, \rho,\\
	\forall X \; \big[\forall x \in X\; \exists T_{A}, T_{B}\; (vertex(T_{A}, T_{B}) \,\land\, T_{A} \cup T_{B} \subseteq X \,\land\, x \in T_{A} \cup T_{B})\big] \rightarrow \rho.
\end{gather*}

\smallskip
Finally, we need to replace every predicate $E(x,y)$ inside $\varphi$.
In order to do that, we first show that it is possible to define a predicate identifying a single block inside $\pi$.
Observe that depending on the shape of the path, each block might end up transformed in one of four possible ways.
These symmetries are generated by reversal and complement.
We define the \emph{row orientation} of a block $B_i$ to be 1 if the bottommost atomic pair is $Z^i_1$, and -1 otherwise.
Similarly, we define the \emph{column orientation} of a block $B_i$ to be 1 if the leftmost atomic pair is $Z^i_1$, and -1 otherwise.

We define a predicate $block^{(r,c)}(S)$ for every $r,c \in \{-1,1\}$ that evaluates true if and only if $S$ is the point set of a block $B_i$ with row orientation $r$ and column orientation $c$.
Instead of writing the full technical definition, we list individual properties that are sufficient and each of them is easily seen to be expressible by an MSO sentence.
Moreover, we only list the properties defining $block^{(1,1)}(S)$ as the other three possibilities are symmetric.
The following must hold for a set $S$ satisfying $block^{(1,1)}(S)$:
\begin{enumerate}[label=(\roman*)]
	\item $S$ is equal to $[p,q]_{\prec_1} \cap [r,s]_{\prec_2}$ for some points $p, q, r$ and $s$,
	\item $S$ is an increasing point set,
	\item $|S \cap T_{Z_1}| = |S\cap T_{Z_2}| = |S \cap T_{Z_3}| = 2$, and
	\item the bottommost 2 points of $S$ belong to $T_{Z_1}$ while the topmost 2 points belong to $T_{Z_3}$.
\end{enumerate}

When we have a set variable $S$ representing a single block, we can test whether there is an edge between vertices $x$ and $y$ (represented by variables $T_{X_x}$, $T_{Y_x}$, $T_{X_y}$ and $T_{Y_y}$) represented inside the block $S$ by verifying that
\begin{enumerate}[label=(\roman*)]
	\item there are exactly two points lying between $T_{X_x} \cap S$ and $T_{Y_x} \cap S$ both horizontally and vertically, and
	\item there is exactly one point lying between $T_{X_y} \cap S$ and $T_{Y_y} \cap S$ both horizontally and vertically.
\end{enumerate}
We can clearly encode this as an MSO predicate $block\_edge(T_{X_x}, T_{Y_x}, T_{X_y}, T_{Y_y}, S)$.
Finally, we can replace every predicate $E(x,y)$ inside $\varphi$ with the predicate $edge(T_{X_x},\allowbreak T_{Y_x},\allowbreak T_{X_y},\allowbreak T_{Y_y})$ defined as
\begin{multline*}
	edge(T_{X_x}, T_{Y_x}, T_{X_y}, T_{Y_y}) = \exists S \;\bigvee_{\mathclap{r,c \in \{-1,1\}}}  block^{(r,c)}(S) \land\; block\_edge(T_{X_x}, T_{Y_x}, T_{X_y}, T_{Y_y}, S).
\end{multline*}

To wrap up, observe that $\pi$ is a permutation of length $O(n^2)$ and the length of $\psi$ is $O(|\varphi|)$ as promised.
Moreover, it is clear from the construction of $\pi$ and $\psi$ that $G \models \varphi$ if and only if $\pi \models \psi$.
\end{proof}

Let us remark that we immediately obtain the dichotomy for hardness of MSO model checking for classes defined by avoiding a single fixed pattern.
Recall that Jelínek et al.~\cite{Jelinek2021b} showed that $\Av(\sigma)$ has the poly-time computable LPP for any $\sigma$ not symmetric to any of $1$, $12$ or $132$ and thus, MSO model checking is hard by Theorem~\ref{thm:mso-checking-np-hard}.
On the other hand, the classes $\Av(1)$ and $\Av(12)$ are trivial, and the bounded tree-width of $\Av(132)$ follows from the results of Ahal and Rabinovich~\cite{Ahal2008}.
As a result, MSO model checking inside a permutation class $\Av(\sigma)$ is decidable by an FPT algorithm parameterized by the length of the formula if $\sigma$ is symmetric to $1, 12$ or $132$; and
otherwise, MSO model checking is as hard as in general graphs.



\bibliography{bibliography}

\end{document}

%% file: intro.tex
A classical result of Courcelle~\cite{Courcelle} states that any graph property expressible in 
monadic second-order logic (MSO) can be tested efficiently on graphs of bounded tree-width. This 
theorem has inspired further results of a similar form; thus, Courcelle, Makowsky and 
Rotics~\cite{Courcelle2000} showed a similar result for graphs of bounded clique-width using a 
weaker form of MSO, while first-order properties are known to be tractable on monotone 
nowhere-dense classes~\cite{gks}, as well as classes of bounded twin-width~\cite{Bonnet2020}.

As a partial converse to Courcelle's theorem, Kreutzer and Tazari~\cite{kt} have shown that under 
plausible complexity assumptions, MSO-definable properties cannot be efficiently tested on any 
monotone graph class whose tree-width is not bounded from above by a polylogarithmic function. A key 
component in the proof of this result, as well as of another similar result by Ganian et 
al.~\cite{ganian}, is a structural characterization of graphs of large tree-with, ultimately derived 
from the grid theorem of Robertson and Seymour~\cite{rs} stating that a class of graphs has 
unbounded tree-width if and only if its graphs contain arbitrarily large grids as minors. 

In this paper, we focus on hereditary classes of permutations. As is common in model-theoretic 
contexts~\cite{bofo,cam_per,Albert2020}, we 
represent a permutation as a relational structure equipped with a pair of binary relations, each of 
them representing a linear order. Ahal and Rabinovich~\cite{Ahal2008} have described a natural way 
to define tree-width for permutations. Their main motivation was to address the complexity of the 
decision problem known as Permutation Pattern Matching (PPM), whose goal is to determine whether a 
permutation $\tau$ (the \emph{text}) contains another permutation $\pi$ (the \emph{pattern}) as 
substructure. While PPM was shown to be NP-complete by Bose et al.~\cite{Bose1998}, the 
results of Ahal and Rabinovich imply that PPM is polynomial when the pattern is restricted to a 
fixed class of bounded tree-width. 

Although the complexity of restricted instances of PPM has recently attracted considerable 
attention~\cite{Albert2016,GV09_321,BerendsohnMs,Berendsohn2021,Jelinek2017,Jelinek2020,Jelinek2021,Jelinek2021b}, we still do not know 
whether PPM can be tractable when patterns are restricted to a hereditary permutation class of 
unbounded tree-width. One difficulty stems from the fact that, unlike in the case of graph classes, 
we do not have a suitable structural characterization of permutation classes of unbounded tree-width 
analogous to the grid theorem for graphs. A promising approach towards such a characterisation is 
based on the concept of \emph{long path property} (or LPP) of a permutation class. It is known that 
a permutation class that has the LPP must have unbounded tree-width~\cite{Jelinek2021b}; however, it 
is not known whether the converse holds as well. Nevertheless, LPP has played (sometimes implicitly) 
a key role in several hardness results on 
PPM~\cite{BerendsohnMs,Jelinek2017,Jelinek2020,Jelinek2021b}.

Recently, Albert et al.~\cite{Albert2020} studied the expressive power of 
permutation properties definable in first order logic. For instance, they considered 
two types of FO logic, one based on the above-mentioned representation of permutations by two linear 
orders (which they call \TOTO{} - the theory of two orders), and another based on a representation by a directed
graph formed by disjoint cycles (\TOOB{} - the theory of one bijection). They show that the two theories are
incomparable in terms of the expressive power of their FO formulas. In particular, they show that the
property of having a fixed point is FO-definable in \TOOB, but not in \TOTO. In this paper, we deal exclusively 
with the more usual formalism of \TOTO.

Two notable algorithmic applications of logic on permutations have appeared in recent years.
Bonnet et al.~\cite{Bonnet2020} have shown that FO model checking (for \TOTO) is tractable inside any proper permutation class.
In a recent paper, Braunfeld~\cite{Braunfeld2023} utilized monadic second order logic of permutations to show that the basis and generating function of any geometric grid class of permutations are algorithmically computable.

\paragraph{Our results.}
In this paper, we focus on the expressive power and algorithmic tractability of MSO-definable 
properties of permutations.

In Section~\ref{sec:expr}, we explore the expressive power of MSO logic on permutations and show that MSO can express modular counting properties of permutation statistics, such as divisibility of the number of inversions or major index.
Moreover, we obtain general examples of MSO-definable properties that are not FO-definable, as well as an example of a natural property (to have a fixed point) which is not MSO-definable.

In Section~\ref{sec:model-checking}, we focus on the algorithmic tractability.
We show that MSO model checking is in FPT on any permutation class of bounded tree-width. In 
contrast, we show that on any permutation class with a suitable effective version of the long path property, MSO model 
checking is at least as hard as MSO model checking on general graphs.
As a corollary, we obtain a dichotomy for classes avoiding a single pattern -- there is an FPT algorithm for MSO model checking on permutations avoiding any of the patterns $1$, $12$, $132$ or their symmetries; otherwise model checking remains as hard as on general graphs.

%% file: arxiv.bbl
\begin{thebibliography}{45}
\providecommand{\natexlab}[1]{#1}
\providecommand{\url}[1]{\texttt{#1}}
\expandafter\ifx\csname urlstyle\endcsname\relax
  \providecommand{\doi}[1]{doi: #1}\else
  \providecommand{\doi}{doi: \begingroup \urlstyle{rm}\Url}\fi

\bibitem[Ahal and Rabinovich(2008)]{Ahal2008}
S.~Ahal and Y.~Rabinovich.
\newblock On complexity of the subpattern problem.
\newblock \emph{SIAM Journal on Discrete Mathematics}, 22\penalty0 (2):\penalty0 629--649, 2008.
\newblock \doi{10.1137/S0895480104444776}.

\bibitem[Albert and Jel\'{\i}nek(2016)]{Albert2016a}
M.~Albert and V.~Jel\'{\i}nek.
\newblock Unsplittable classes of separable permutations.
\newblock \emph{Electron. J. Combin.}, 23\penalty0 (2):\penalty0 Paper 2.49, 20, 2016.
\newblock \doi{10.37236/6115}.

\bibitem[Albert et~al.(2016)Albert, Lackner, Lackner, and Vatter]{Albert2016}
M.~Albert, M.-L. Lackner, M.~Lackner, and V.~Vatter.
\newblock The complexity of pattern matching for 321-avoiding and skew-merged permutations.
\newblock \emph{Discrete Mathematics \& Theoretical Computer Science. DMTCS.}, 18\penalty0 (2):\penalty0 Paper No. 11, 17, 2016.
\newblock \doi{10.46298/dmtcs.1308}.

\bibitem[Albert et~al.(2019)Albert, Pantone, and Vatter]{Albert2019}
M.~Albert, J.~Pantone, and V.~Vatter.
\newblock On the growth of merges and staircases of permutation classes.
\newblock \emph{Rocky Mountain J. Math.}, 49\penalty0 (2):\penalty0 355--367, 2019.
\newblock ISSN 0035-7596.
\newblock \doi{10.1216/RMJ-2019-49-2-355}.

\bibitem[Albert et~al.(2020)Albert, Bouvel, and F\'{e}ray]{Albert2020}
M.~Albert, M.~Bouvel, and V.~F\'{e}ray.
\newblock Two first-order logics of permutations.
\newblock \emph{J. Combin. Theory Ser. A}, 171:\penalty0 105158, 46, 2020.
\newblock ISSN 0097-3165.
\newblock \doi{10.1016/j.jcta.2019.105158}.

\bibitem[Albert(2007)]{Albert07}
M.~H. Albert.
\newblock On the length of the longest subsequence avoiding an arbitrary pattern in a random permutation.
\newblock \emph{Random Structures Algorithms}, 31\penalty0 (2):\penalty0 227--238, 2007.
\newblock ISSN 1042-9832.
\newblock \doi{10.1002/rsa.20140}.

\bibitem[Babson and Steingr\'{i}msson(2000)]{Babson2000}
E.~Babson and E.~Steingr\'{i}msson.
\newblock Generalized permutation patterns and a classification of the {M}ahonian statistics.
\newblock \emph{S\'{e}minaire Lotharingien de Combinatoire}, 44:\penalty0 Art. B44b, 18, 2000.

\bibitem[Berendsohn(2019)]{BerendsohnMs}
B.~A. Berendsohn.
\newblock Complexity of permutation pattern matching.
\newblock Master's thesis, Freie Universit\"at Berlin, 2019.

\bibitem[Berendsohn et~al.(2021)Berendsohn, Kozma, and Marx]{Berendsohn2021}
B.~A. Berendsohn, L.~Kozma, and D.~Marx.
\newblock Finding and counting permutations via {CSP}s.
\newblock \emph{Algorithmica. An International Journal in Computer Science}, 83\penalty0 (8):\penalty0 2552--2577, 2021.
\newblock ISSN 0178-4617.
\newblock \doi{10.1007/s00453-021-00812-z}.

\bibitem[Bevan et~al.(2017)Bevan, Brignall, Price, and Pantone]{Bevan2017}
D.~Bevan, R.~Brignall, A.~E. Price, and J.~Pantone.
\newblock Staircases, dominoes, and the growth rate of 1324-avoiders.
\newblock \emph{Electronic Notes in Discrete Mathematics}, 61:\penalty0 123--129, 2017.
\newblock ISSN 1571-0653.
\newblock \doi{10.1016/j.endm.2017.06.029}.
\newblock The European Conference on Combinatorics, Graph Theory and Applications (EUROCOMB'17).

\bibitem[B\'{o}na(2014)]{Bona2014}
M.~B\'{o}na.
\newblock A new upper bound for 1324-avoiding permutations.
\newblock \emph{Combin. Probab. Comput.}, 23\penalty0 (5):\penalty0 717--724, 2014.
\newblock ISSN 0963-5483.
\newblock \doi{10.1017/S0963548314000091}.

\bibitem[Bonnet et~al.(2020)Bonnet, Kim, Thomass{\'{e}}, and Watrigant]{Bonnet2020}
{\'{E}}.~Bonnet, E.~J. Kim, S.~Thomass{\'{e}}, and R.~Watrigant.
\newblock Twin-width {I:} tractable {FO} model checking.
\newblock In S.~Irani, editor, \emph{61st {IEEE} Annual Symposium on Foundations of Computer Science, {FOCS} 2020, Durham, NC, USA, November 16-19, 2020}, pages 601--612. {IEEE}, 2020.
\newblock \doi{10.1109/FOCS46700.2020.00062}.

\bibitem[Bose et~al.(1998)Bose, Buss, and Lubiw]{Bose1998}
P.~Bose, J.~F. Buss, and A.~Lubiw.
\newblock Pattern matching for permutations.
\newblock \emph{Information Processing Letters}, 65\penalty0 (5):\penalty0 277--283, 1998.
\newblock \doi{10.1016/S0020-0190(97)00209-3}.

\bibitem[B{\"o}ttcher and Foniok(2013)]{bofo}
J.~B{\"o}ttcher and J.~Foniok.
\newblock {R}amsey properties of permutations.
\newblock \emph{Electron. J. Comb.}, 20\penalty0 (1):\penalty0 \#P2, 2013.
\newblock \doi{10.37236/2978}.

\bibitem[Br\"{a}nd\'{e}n and Claesson(2011)]{Branden2011}
P.~Br\"{a}nd\'{e}n and A.~Claesson.
\newblock Mesh patterns and the expansion of permutation statistics as sums of permutation patterns.
\newblock \emph{Electron. J. Combin.}, 18\penalty0 (2):\penalty0 Paper 5, 14, 2011.
\newblock ISSN 1077-8926.
\newblock \doi{10.37236/2001}.

\bibitem[Braunfeld(2025)]{Braunfeld2023}
S.~Braunfeld.
\newblock Decidability in geometric grid classes of permutations.
\newblock \emph{Proceedings of the American Mathematical Society}, 153\penalty0 (03):\penalty0 987--1000, 2025.
\newblock \doi{10.1090/proc/17083}.

\bibitem[B\"{u}chi(1960)]{Buchi1960}
J.~R. B\"{u}chi.
\newblock Weak second-order arithmetic and finite automata.
\newblock \emph{Z. Math. Logik Grundlagen Math.}, 6:\penalty0 66--92, 1960.
\newblock ISSN 0044-3050.
\newblock \doi{10.1002/malq.19600060105}.

\bibitem[Cameron(2002)]{cam_per}
P.~J. Cameron.
\newblock Homogeneous permutations.
\newblock \emph{Electron. J. Comb.}, 9\penalty0 (2):\penalty0 \#R2, 2002.
\newblock \doi{10.37236/1674}.

\bibitem[Claesson et~al.(2012)Claesson, Jel\'{\i}nek, and Steingr\'{\i}msson]{Claesson2012}
A.~Claesson, V.~Jel\'{\i}nek, and E.~Steingr\'{\i}msson.
\newblock Upper bounds for the {S}tanley-{W}ilf limit of 1324 and other layered patterns.
\newblock \emph{J. Combin. Theory Ser. A}, 119\penalty0 (8):\penalty0 1680--1691, 2012.
\newblock ISSN 0097-3165.
\newblock \doi{10.1016/j.jcta.2012.05.006}.

\bibitem[Cormen et~al.(2009)Cormen, Leiserson, Rivest, and Stein]{Cormen2009}
T.~H. Cormen, C.~E. Leiserson, R.~L. Rivest, and C.~Stein.
\newblock \emph{Introduction to Algorithms, 3rd Edition}.
\newblock {MIT} Press, 2009.
\newblock ISBN 978-0-262-03384-8.
\newblock URL \url{http://mitpress.mit.edu/books/introduction-algorithms}.

\bibitem[Courcelle(1990)]{Courcelle}
B.~Courcelle.
\newblock The monadic second-order logic of graphs. {I}. {R}ecognizable sets of finite graphs.
\newblock \emph{Inform. and Comput.}, 85\penalty0 (1):\penalty0 12--75, 1990.
\newblock ISSN 0890-5401.
\newblock \doi{10.1016/0890-5401(90)90043-H}.

\bibitem[Courcelle(1996)]{Courcelle96}
B.~Courcelle.
\newblock The monadic second-order logic of graphs {X:} linear orderings.
\newblock \emph{Theor. Comput. Sci.}, 160\penalty0 (1{\&}2):\penalty0 87--143, 1996.
\newblock \doi{10.1016/0304-3975(95)00083-6}.

\bibitem[Courcelle and Engelfriet(2012)]{Courcelle2012}
B.~Courcelle and J.~Engelfriet.
\newblock \emph{Graph Structure and Monadic Second-Order Logic - {A} Language-Theoretic Approach}, volume 138 of \emph{Encyclopedia of mathematics and its applications}.
\newblock Cambridge University Press, 2012.
\newblock ISBN 978-0-521-89833-1.
\newblock URL \url{http://www.cambridge.org/fr/knowledge/isbn/item5758776/?site\_locale=fr\_FR}.

\bibitem[Courcelle et~al.(2000)Courcelle, Makowsky, and Rotics]{Courcelle2000}
B.~Courcelle, J.~A. Makowsky, and U.~Rotics.
\newblock Linear time solvable optimization problems on graphs of bounded clique-width.
\newblock \emph{Theory Comput. Syst.}, 33\penalty0 (2):\penalty0 125--150, 2000.
\newblock ISSN 1432-4350.
\newblock \doi{10.1007/s002249910009}.

\bibitem[Ehrenfeucht(1960/61)]{Ehrenfeucht1960}
A.~Ehrenfeucht.
\newblock An application of games to the completeness problem for formalized theories.
\newblock \emph{Fund. Math.}, 49:\penalty0 129--141, 1960/61.
\newblock ISSN 0016-2736.
\newblock \doi{10.4064/fm-49-2-129-141}.

\bibitem[Elberfeld et~al.(2016)Elberfeld, Frickenschmidt, and Grohe]{Elberfeld2016}
M.~Elberfeld, M.~Frickenschmidt, and M.~Grohe.
\newblock Order invariance on decomposable structures.
\newblock In M.~Grohe, E.~Koskinen, and N.~Shankar, editors, \emph{Proceedings of the 31st Annual {ACM/IEEE} Symposium on Logic in Computer Science, {LICS} '16, New York, NY, USA, July 5-8, 2016}, pages 397--406. {ACM}, 2016.
\newblock \doi{10.1145/2933575.2934517}.

\bibitem[Elgot(1961)]{Elgot1961}
C.~C. Elgot.
\newblock Decision problems of finite automata design and related arithmetics.
\newblock \emph{Trans. Amer. Math. Soc.}, 98:\penalty0 21--51, 1961.
\newblock ISSN 0002-9947.
\newblock \doi{10.2307/1993511}.

\bibitem[Fra\"{i}ss\'{e}(1954)]{Fraisse1954}
R.~Fra\"{i}ss\'{e}.
\newblock Sur quelques classifications des syst\`emes de relations.
\newblock \emph{Publ. Sci. Univ. Alger. S\'{e}r. A}, 1:\penalty0 35--182 (1955), 1954.

\bibitem[Ganian et~al.(2014)Ganian, Hlin\v{e}n\'{y}, Langer, Obdr\v{z}\'{a}lek, Rossmanith, and Sikdar]{ganian}
R.~Ganian, P.~Hlin\v{e}n\'{y}, A.~Langer, J.~Obdr\v{z}\'{a}lek, P.~Rossmanith, and S.~Sikdar.
\newblock Lower bounds on the complexity of {MSO$_1$} model-checking.
\newblock \emph{J. Comput. System Sci.}, 80\penalty0 (1):\penalty0 180--194, 2014.
\newblock ISSN 0022-0000.
\newblock \doi{10.1016/j.jcss.2013.07.005}.

\bibitem[Ganzow and Rubin(2008)]{Ganzow2008}
T.~Ganzow and S.~Rubin.
\newblock Order-invariant {MSO} is stronger than counting {MSO} in the finite.
\newblock In S.~Albers and P.~Weil, editors, \emph{{STACS} 2008, 25th Annual Symposium on Theoretical Aspects of Computer Science, Bordeaux, France, February 21-23, 2008, Proceedings}, volume~1 of \emph{LIPIcs}, pages 313--324. Schloss Dagstuhl - Leibniz-Zentrum f{\"{u}}r Informatik, Germany, 2008.
\newblock \doi{10.4230/LIPICS.STACS.2008.1353}.

\bibitem[Gr\"{a}del et~al.(2007)Gr\"{a}del, Kolaitis, Libkin, Marx, Spencer, Vardi, Venema, and Weinstein]{Gradel2007}
E.~Gr\"{a}del, P.~G. Kolaitis, L.~Libkin, M.~Marx, J.~Spencer, M.~Y. Vardi, Y.~Venema, and S.~Weinstein.
\newblock \emph{Finite model theory and its applications}.
\newblock Texts in Theoretical Computer Science. An EATCS Series. Springer, Berlin, 2007.
\newblock ISBN 978-3-540-00428-8.

\bibitem[Grohe et~al.(2017)Grohe, Kreutzer, and Siebertz]{gks}
M.~Grohe, S.~Kreutzer, and S.~Siebertz.
\newblock Deciding first-order properties of nowhere dense graphs.
\newblock \emph{J. ACM}, 64\penalty0 (3):\penalty0 Art. 17, 32, 2017.
\newblock ISSN 0004-5411.
\newblock \doi{10.1145/3051095}.

\bibitem[Guillemot and Vialette(2009)]{GV09_321}
S.~Guillemot and S.~Vialette.
\newblock Pattern matching for 321-avoiding permutations.
\newblock In \emph{Algorithms and computation}, volume 5878 of \emph{Lecture Notes in Comput. Sci.}, pages 1064--1073. Springer, Berlin, 2009.
\newblock \doi{10.1007/978-3-642-10631-6\_107}.

\bibitem[Jel\'{i}nek and Kyn\v{c}l(2017)]{Jelinek2017a}
V.~Jel\'{i}nek and J.~Kyn\v{c}l.
\newblock Hardness of permutation pattern matching.
\newblock In \emph{Proceedings of the {T}wenty-{E}ighth {A}nnual {ACM}-{SIAM} {S}ymposium on {D}iscrete {A}lgorithms}, pages 378--396. SIAM, Philadelphia, PA, 2017.
\newblock \doi{10.1137/1.9781611974782.24}.

\bibitem[Jel\'{\i}nek and Opler(2017)]{Jelinek2017}
V.~Jel\'{\i}nek and M.~Opler.
\newblock Splittability and 1-amalgamability of permutation classes.
\newblock \emph{Discrete Math. Theor. Comput. Sci.}, 19\penalty0 (2):\penalty0 Paper No. 4, 14, 2017.
\newblock \doi{10.23638/DMTCS-19-2-4}.

\bibitem[Jel\'{\i}nek and Valtr(2015)]{Jelinek2015}
V.~Jel\'{\i}nek and P.~Valtr.
\newblock Splittings and {R}amsey properties of permutation classes.
\newblock \emph{Adv. in Appl. Math.}, 63:\penalty0 41--67, 2015.
\newblock ISSN 0196-8858.
\newblock \doi{10.1016/j.aam.2014.10.003}.

\bibitem[Jel{\'{i}}nek et~al.(2020)Jel{\'{i}}nek, Opler, and Pek{\'{a}}rek]{Jelinek2020}
V.~Jel{\'{i}}nek, M.~Opler, and J.~Pek{\'{a}}rek.
\newblock A complexity dichotomy for permutation pattern matching on grid classes.
\newblock In J.~Esparza and D.~Kr{\'{a}}l', editors, \emph{45th International Symposium on Mathematical Foundations of Computer Science, {MFCS} 2020, August 24-28, 2020, Prague, Czech Republic}, volume 170 of \emph{LIPIcs}, pages 52:1--52:18. Schloss Dagstuhl - Leibniz-Zentrum f{\"{u}}r Informatik, 2020.
\newblock \doi{10.4230/LIPIcs.MFCS.2020.52}.

\bibitem[Jel\'{i}nek et~al.(2021)Jel\'{i}nek, Opler, and Pek\'{a}rek]{Jelinek2021}
V.~Jel\'{i}nek, M.~Opler, and J.~Pek\'{a}rek.
\newblock {Griddings of Permutations and Hardness of Pattern Matching}.
\newblock In F.~Bonchi and S.~J. Puglisi, editors, \emph{46th International Symposium on Mathematical Foundations of Computer Science (MFCS 2021)}, volume 202 of \emph{LIPIcs}, pages 65:1--65:22. Schloss Dagstuhl -- Leibniz-Zentrum f{\"u}r Informatik, 2021.
\newblock ISBN 978-3-95977-201-3.
\newblock \doi{10.4230/LIPIcs.MFCS.2021.65}.

\bibitem[Jel{\'\i}nek et~al.(2021)Jel{\'\i}nek, Opler, and Pek\'{a}rek]{Jelinek2021b}
V.~Jel{\'\i}nek, M.~Opler, and J.~Pek\'{a}rek.
\newblock {Long Paths Make Pattern-Counting Hard, and Deep Trees Make It Harder}.
\newblock In P.~A. Golovach and M.~Zehavi, editors, \emph{16th International Symposium on Parameterized and Exact Computation (IPEC 2021)}, volume 214 of \emph{Leibniz International Proceedings in Informatics (LIPIcs)}, pages 22:1--22:17, Dagstuhl, Germany, 2021. Schloss Dagstuhl -- Leibniz-Zentrum f{\"u}r Informatik.
\newblock ISBN 978-3-95977-216-7.
\newblock \doi{10.4230/LIPIcs.IPEC.2021.22}.

\bibitem[Jel{\'i}nek et~al.(2024)Jel{\'i}nek, Opler, and Valtr]{JOV}
V.~Jel{\'i}nek, M.~Opler, and P.~Valtr.
\newblock Generalized coloring of permutations.
\newblock \emph{Algorithmica}, Mar 2024.
\newblock ISSN 1432-0541.
\newblock \doi{10.1007/s00453-024-01220-9}.

\bibitem[Kreutzer and Tazari(2010)]{kt}
S.~Kreutzer and S.~Tazari.
\newblock Lower bounds for the complexity of monadic second-order logic.
\newblock In \emph{25th {A}nnual {IEEE} {S}ymposium on {L}ogic in {C}omputer {S}cience {LICS} 2010}, pages 189--198. {IEEE} {C}omputer {S}oc., {L}os {A}lamitos, {CA}, 2010.
\newblock \doi{10.1109/LICS.2010.39}.

\bibitem[Murphy(2003)]{Murphy2003}
M.~M. Murphy.
\newblock \emph{Restricted permutations, antichains, atomic classes and stack sorting}.
\newblock PhD thesis, University of St Andrews, 2003.

\bibitem[Robertson and Seymour(1986)]{rs}
N.~Robertson and P.~D. Seymour.
\newblock Graph minors. {V}. {E}xcluding a planar graph.
\newblock \emph{Journal of Combinatorial Theory, Series B}, 41\penalty0 (1):\penalty0 92--114, 1986.
\newblock \doi{10.1016/0095-8956(86)90030-4}.

\bibitem[Trakhtenbrot(1962)]{Trakhtenbrot1962}
B.~A. Trakhtenbrot.
\newblock Finite automata and the logic of one-place predicates.
\newblock \emph{Sibirsk. Mat. \v{Z}.}, 3:\penalty0 103--131, 1962.
\newblock ISSN 0037-4474.

\bibitem[Vatter(2016)]{Vatter2016}
V.~Vatter.
\newblock An {{E}rd\H{o}s}--{H}ajnal analogue for permutation classes.
\newblock \emph{Discrete Math. Theor. Comput. Sci.}, 18\penalty0 (2):\penalty0 Paper No. 4, 5pp, 2016.
\newblock \doi{10.46298/dmtcs.1328}.

\end{thebibliography}
